\documentclass[11pt, a4paper]{amsart}
\usepackage{amssymb}
\usepackage{graphicx}
\usepackage{color}
\newtheorem{lemm}{Lemma}[section]
\newtheorem{coro}{Corollary}[section]
\newtheorem{prop}{Proposition}[section]
\newtheorem{theo}{Theorem}[section]
\theoremstyle{definition}
\newtheorem{defi}{Definition}[section]
\newtheorem{rema}{Remark}[section]

\numberwithin{equation}{section}

\def\Bbb{\mathbb}
\def\fb{\partial\{u>0\}}
\def\uep{u^\varepsilon}
\def\ep{\varepsilon}
\def\pep{P_\varepsilon}

\def\a{\alpha}

\def\fint{\operatorname {--\!\!\!\!\!\int\!\!\!\!\!--}}

\def\uepj{u^{\varepsilon_j}}
\def\pepj{P_{\varepsilon_j}}

\def\rn1{\Bbb R^{N+1}}
\def\rn{\Bbb R^{N}}
\def\epjn{\varepsilon_{j_n}}

\def\uepjnln{(u^{\varepsilon_{j_n}})_{\lambda_n}}

\def\Ln{\lambda_n}

\def\epj{\varepsilon_j}

\def\fepj{f^{{\varepsilon}_j}}
\def\fep{f^{{\varepsilon}}}

\def\vep{v^\ep}

\def\a*{\alpha^*_\ep}

\def\bepj{B_{\varepsilon_j}}

\def\R{\mathbb R}

\def\lstar{\lambda^*}
\def\lone{\lambda_{\min}}
\def\ltwo{\lambda_{\max}}

\begin{document}
\title[An inhomogeneous singular perturbation problem for the $p(x)$-Laplacian]{An inhomogeneous singular perturbation problem \\ for the $p(x)$-Laplacian}
\author[Claudia Lederman]{Claudia Lederman}
\author[Noemi Wolanski]{Noemi  Wolanski}
\address{IMAS - CONICET and Departamento  de
Ma\-te\-m\'a\-ti\-ca, Facultad de Ciencias Exactas y Naturales,
Universidad de Buenos Aires, (1428) Buenos Aires, Argentina.}
\email[Claudia Lederman]{clederma@dm.uba.ar} \email[Noemi
Wolanski]{wolanski@dm.uba.ar}
\thanks{Supported by the Argentine Council of Research CONICET under the project PIP625, Res. 960/12,  UBACYT 20020100100496 and
ANPCyT PICT 2012-0153.}

 \dedicatory{Dedicated to our dear friend and colleague Juan Luis V\'azquez on the
 occasion of his 70th birthday}

\maketitle

\begin{abstract}
In this paper we study  the  following singular perturbation problem for the $p_\ep(x)$-Laplacian:
\begin{equation}
\label{eq}\tag{$P_\ep(\fep, p_\ep)$}
\Delta_{p_\ep(x)}\uep:=\mbox{div}(|\nabla \uep(x)|^{p_\ep(x)-2}\nabla
\uep)={\beta}_{\varepsilon}(\uep)+\fep, \quad u^{\ep}\geq 0,
\end{equation}
 where $\ep>0$,
${\beta}_{\varepsilon}(s)={1 \over \varepsilon} \beta({s \over
\varepsilon})$, with $\beta$  a  Lipschitz  function satisfying
$\beta>0$ in $(0,1)$, $\beta\equiv 0$ outside $(0,1)$ and $\int
\beta(s)\, ds=M$. The functions $\uep$, $\fep$ and $p_\ep$ are uniformly
bounded.
We prove uniform Lipschitz regularity, we pass to the limit $(\ep\to 0)$ and
we show that, under  suitable assumptions, limit functions are weak solutions to the free
boundary problem: $u\ge0$ and
\begin{equation}
\label{fbp-px}\tag{$P(f,p,{\lambda}^*)$}
\begin{cases}
\Delta_{p(x)}u= f & \mbox{in }\{u>0\}\\
u=0,\ |\nabla u| = \lambda^*(x) & \mbox{on }\partial\{u>0\}
\end{cases}
\end{equation}
with $\lambda^*(x)=\Big(\frac{p(x)}{p(x)-1}\,M\Big)^{1/p(x)}$,
$p=\lim p_\ep$ and $f=\lim \fep$.

In \cite{LW4} we prove that
 the free boundary of a weak solution
is a $C^{1,\alpha}$ surface near flat free boundary points. This result applies, in particular, to the limit functions studied in this paper.
\end{abstract}

\bigskip


\begin{section}{Introduction}
\label{sect-intro}

Singular perturbation problems of the form
\begin{equation}\label{ph}
L\uep = \beta_\ep(\uep)
\end{equation}
with ${\beta}_{\varepsilon}(s)={1 \over \varepsilon} \beta({s
\over \varepsilon})$, $\beta$ nonnegative, smooth and supported on
$[0,1]$ and $L$ an elliptic or parabolic second order differential operator
have been widely studied due to their appearance in different
contexts. One of its main application being to flame propagation.
See \cite{BCN,BL,CV,ZFK} and also the excellent survey by J. L. V\'azquez \cite{V}.

A natural generalization is the consideration of inhomogeneous problems
\begin{equation}\label{pih}
L\uep=\beta_\ep(\uep)+\fep
\end{equation}
with $\fep$ uniformly bounded independently of $\ep$.  The inhomogeneous terms may represent sources as well as nonlocal effects, when the family $\uep$ is uniformly bounded (see \cite{LW2}).

Problem \eqref{ph} was first studied for a linear uniformly elliptic operator $L$ by Berestycki, Caffarelli and Nirenberg in \cite{BCN} and then for the heat equation by Caffarelli and V\'azquez in \cite{CV}. The two phase  case for the heat equation was studied by Caffarelli and the authors in \cite{CLW1,CLW2}.
A natural question is the identification of the limiting problem as $\ep\to0$. To this end, estimates uniform in $\ep$ are needed. These two questions were the object of the above mentioned articles \cite{BCN, CV, CLW1, CLW2}.

For the inhomogeneous problem \eqref{pih} and $L=\Delta$  or $L=\Delta-\partial_t$ these questions were settled in \cite{LW2,LW3}.

The homogeneous problem  \eqref{ph} in the case of the $p$-Laplacian was considered in \cite{DPS} and then, for more general operators with power like growth  in \cite{MW2}.
Uniform estimates for the inhomogeneous problem \eqref{pih} and the $p$-Laplacian were obtained in \cite{MoWa1}. Additional results
for these type of problems were obtained in \cite{AW,LO,LW1,MoWa1,RT,W}.

\bigskip

In this paper we study the case where the operator $L$ is the $p_\ep(x)$-Laplacian, defined as
\begin{equation*}
\Delta_{p_\ep(x)}u:=\mbox{div}(|\nabla u(x)|^{p_\ep(x)-2}\nabla u),
\end{equation*}
that extends the Laplacian, where $p_\ep(x)\equiv 2$, and the
$p$-Laplacian,
 where $p_\ep(x)\equiv p$ with $1<p<\infty$. The $p(x)$-Laplacian has been
 used in the modeling of electrorheological fluids (\cite{R}) and
 in image processing (\cite{AMS}, \cite{CLR}).

 We consider the inhomogenous problem \eqref{pih}  but we remark that this singular perturbation problem for the  $p_\ep(x)$-Laplacian  had not been studied even in the homogeneous case \eqref{ph}. Moreover, the identification of the limiting problem in the inhomogeneous case had not been done even for $p_\ep(x)\equiv p$.

 As stated above, this singular perturbation problem may model flame propagation in a fluid with electromagnetic sensitivity. Hence its interest
 from a modeling point of view. On the other hand, the presence of a variable exponent $p_\ep(x)$ and a right hand side $f_\ep(x)$ brings
 new mathematical difficulties, that can be found scattered all along this paper, that were not present in the constant case $p_\ep(x)\equiv p$. An
 important tool we use is the Harnack Inequality for the inhomogeneous $p(x)$-Laplacian that we recently proved in \cite{Wo}.

\medskip

More precisely, in this paper we study the following singular perturbation
 problem for the $p_\ep(x)$-Laplacian:
\begin{equation}
\tag{$P_\ep(\fep, p_\ep)$}
\Delta_{p_\ep(x)}\uep={\beta}_{\varepsilon}(\uep)+\fep, \quad
u^{\ep}\geq 0
\end{equation}
in a domain $\Omega\subset \Bbb R^{N}$. Here $\ep>0$,
${\beta}_{\varepsilon}(s)={1 \over \varepsilon} \beta({s \over
\varepsilon})$, with $\beta$  a  Lipschitz function satisfying
$\beta>0$ in $(0,1)$, $\beta\equiv 0$ outside $(0,1)$ and $\int
\beta(s)\, ds=M$.

We assume that $\uep$, $\fep$  are
uniformly bounded and that $p_\ep$ are uniformly bounded in Lipschitz norm. We prove uniform Lipschitz regularity, we pass
to the limit $(\ep\to 0)$ and we show that, under suitable
assumptions, limit functions are weak solutions to the following
free boundary problem: $u\ge 0$ and
\begin{equation}
\label{bernoulli-px}\tag{$P(f,p,{\lambda}^*)$}
\begin{cases}
\Delta_{p(x)}u= f & \mbox{in }\{u>0\}\\
u=0,\ |\nabla u| = \lambda^*(x) & \mbox{on }\partial\{u>0\}
\end{cases}
\end{equation}
with $\lambda^*(x)=\Big(\frac{p(x)}{p(x)-1}\,M\Big)^{1/p(x)}$, $p=\lim p_\ep$ and
$f=\lim \fep$.

We remark that, in the inhomogeneous case, there are examples of limit functions that are not solutions to the free boundary problem \ref{bernoulli-px}.
These examples were produced with $p_\ep(x)\equiv2$ in \cite{LW2}. Hence, some extra assumptions on the limit functions are needed.

\medskip

In a companion paper \cite{LW4} we study the regularity of the free boundary for weak solutions of {$P(f,p,{\lambda}^*)$} with $p(x)$
Lipschitz and $\lambda^*(x)$ a H\"older continuous function. In \cite{LW4} we show that the free boundary is a $C^{1,\alpha}$ surface near flat free boundary points. This regularity result applies in particular to limits of this singular perturbation problem, under the above mentioned assumptions. These additional assumptions are verified if, for instance, the functions $\uep$ are local minimizers of an energy functional. We prove this last result  in \cite{LW5}. Moreover, in this special case, we show in \cite{LW5} that the set of singular points has zero ${\mathcal H}^{N-1}$ measure.

 In conclusion, in this   first paper  of a series  on the singular perturbation problem $P_\ep(\fep, p_\ep)$ we study the fundamental uniform properties of the solutions and we determine the limiting free boundary problem.

\bigskip

An outline of the paper is as follows: In Section 2 we obtain
uniform bounds of the gradients of solutions to the singular perturbation
problem $P_\ep(\fep, p_\ep)$ (Theorem \ref{estim-lip}). In Section 3 we pass to the limit, in Section 4 we analyze some basic limits and in Section 5 we study the asymptotic
behavior of limit functions. Finally, in Section 6 we
define the notion of weak solution to the free boundary problem \ref{bernoulli-px} and
we show that, under suitable assumptions, limit functions to the singular perturbation $P_\ep(\fep, p_\ep)$ are
weak solutions to the free boundary problem \ref{bernoulli-px} with $\lambda^*(x)=\Big(\frac{p(x)}{p(x)-1}\,M\Big)^{1/p(x)}$ (Theorem \ref{lim=weak}). We also state the result from \cite{LW4} on the regularity of the interface for weak solutions (Theorem \ref{reg-weak}).
 We finish the paper with an appendix where we
 collect some results on variable exponent Sobolev spaces  as well as some other results that are used in the paper.

\begin{subsection}{Assumptions}

\medskip

Throughout the paper we let $\Omega\subset\R^N$ a domain.

\bigskip

\noindent{\bf Assumptions on $p_\ep(x)$ and $p(x)$.} We will
assume that the functions $p_\ep(x)$ verify
\begin{equation}\label{pminmax}
1<p_{\min}\le p_\ep(x)\le p_{\max}<\infty,\qquad x\in\Omega.
\end{equation}
When we are restricted to a ball $B_r$ we use ${p_\ep^r}_-$ and
${p_\ep^r}_+ $ to denote the infimum and the supremum of
$p_\ep(x)$ over $B_r$.

We also assume that $p_\ep(x)$ are continuous up to the boundary and that they have  a uniform modulus of continuity $\omega:\R \to \R$, i.e.
$|p_\ep(x)-p_\ep(y)|\leq \omega(|x-y|)$ if $|x-y|$ is small.

For our main results we need  to assume further that $p_\ep(x)$ are uniformly
Lipschitz continuous in $\Omega$. In that case, we denote by $L$
the Lipschitz constant of $p_\ep(x)$, namely, $\|\nabla
p_\ep\|_{L^{\infty}(\Omega)}\leq L$.

The same assumptions above will be made on the function $p(x)$.

\bigskip

\noindent{\bf Assumptions on $\beta_\ep$.} We will assume that the
functions $\beta_\ep$ are defined by scaling of a single function
$\beta:\Bbb R\to \Bbb R$ satisfying:
\begin{itemize}
\item[i)] $\beta$ is a Lipschitz continuous function, \item [ii)]
$\beta>0$ in $(0,1)$ and $\beta\equiv 0$ otherwise, \item [iii)]
$\int_0^1\beta(s)\,ds=M$.
\end{itemize}
And then  $ \beta_\ep(s):=\frac{1}{\ep}\beta(\frac{s}{\ep}). $
\end{subsection}
\bigskip

\begin{subsection}{Definition of solution to $p(x)$-Laplacian.}

Let $p(x)$ be as above and let $g\in L^{\infty}(\Omega\times\R )$.
We say that $u$ is a solution to
$$
\Delta_{p(x)}u = g(x,u) \ \mbox{ in } \ \Omega
$$
if
$u\in W^{1,p(\cdot)}(\Omega)$ and,  for every  $\varphi \in W_0^{1,p(\cdot)}(\Omega)$, there holds
that
$$
\int_{\Omega} |\nabla u(x)|^{p(x)-2}\nabla u \cdot \nabla
\varphi\, dx =-\int_{\Omega} \varphi\,
g(x,u)\, dx.
$$
By the results in \cite{Wo}, it follows that $u\in L_{\rm loc}^{\infty}(\Omega)$.

\end{subsection}

\bigskip

\begin{subsection}{Notation}\ \ \newline

$\bullet$ $N$ \quad spatial dimension

$\bullet$   $\Omega\cap\partial\{ u>0 \}$ \quad free
boundary

$\bullet$  $|S|$ \quad  $N$-dimensional Lebesgue measure of the
set $S$

$\bullet$ ${\mathcal H}^{N-1}$ \quad  $(N-1)$-dimensional
Hausdorff measure

$\bullet$  $B_r(x_0)$ \quad  open ball of radius $r$ and center
$x_0$

$\bullet$  $B_r$ \quad  open ball of radius $r$ and center
$0$

$\bullet$  $B'_r(x_0)$ \quad  open ball of radius $r$ and center
$x_0$ in $\R^{N-1}$

$\bullet$  $B'_r$ \quad  open ball of radius $r$ and center
$0$ in $\R^{N-1}$

$\bullet$  $\fint_{B_r(x_0)}u= {1\over {|B_r(x_0)|}}
\int_{B_r(x_0)}u\,dx$

$\bullet$  $\fint_{\partial B_r(x_0)}u= {1\over {{\mathcal
H}^{N-1} (\partial B_r(x_0))}} \int_{\partial
B_r(x_0)}u\,d{\mathcal H}^{N-1}$

$\bullet$ $\chi_{{}_S}$ \quad  characteristic function of the set
$S$

$\bullet$ $u^{+}=\text{\rm max}(u,0)$,\quad $u^{-}=\text{\rm
max}(-u,0)$

$\bullet$ $\langle\,\cdot\, ,\,\cdot\,\rangle$\quad scalar product
in $\Bbb R^{N}$

$\bullet$ $B_\ep(s)=\int _0^s\beta_\ep(\tau) \, d\tau$

\end{subsection}

\end{section}

\begin{section}{Uniform bound of the gradient}
\label{sect-sing-pert}

In this section we consider a family of uniformly bounded solutions to the singular perturbation
problem $P_\ep(\fep, p_\ep)$ and prove that their gradients are locally uniformly bounded. Our
main result in the section is the following theorem
\begin{theo}\label{estim-lip} Assume that $1<p_{\min}\le p_\ep(x)\le p_{\max}<\infty$
with  $p_\ep(x)$ Lipschitz continuous and $\|\nabla p_\ep\|_{L^{\infty}}\leq L$, for some $L>0$. Let $u^{\ep}$ be a
solution of
\begin{equation}
\tag{$P_\ep(\fep, p_\ep)$}
\Delta_{p_\ep(x)}u^{\ep}=\beta_{\ep}(u^{\ep})+\fep,  \quad
u^{\ep}\geq 0 \quad \mbox{ in } \Omega,\end{equation} with
$\|u^{\ep}\|_{L^{\infty}(\Omega)}\leq L_1$,
$\|\fep\|_{L^{\infty}(\Omega)}\leq L_2$.
  Then, for $\Omega'\subset\subset \Omega$, we have
$$|\nabla u^{\ep}(x)|\leq C\quad \mbox{ in } \Omega',$$
with $C=C(N,L_1,L_2,\|\beta\|_{L^{\infty}}, p_{\min},p_{\max}, L,
\mbox{dist}(\Omega',\partial\Omega))$, if $\ep\leq
\ep_0(\Omega,\Omega')$.
\end{theo}

An essential tool in the proof will be the following Harnack's
Inequality for the inhomogenous $p(x)$-Laplacian equation, proven
in \cite{Wo}, Theorem 2.1

\begin{theo}\label{har} Assume that
$p(x)$ is locally log-H\"older continuous in $\Omega$. This is, $p(x)$ has locally a modulus of continuity $\omega(r)=C(\log\frac1r)^{-1}$. Let $x_0\in\Omega$
and $0<R\le 1$ such that $\overline{B_{4R}(x_0)}\subset\Omega$. There exists C such that, if $u\in W^{1,p(\cdot)}(\Omega)\cap L^{\infty}(\Omega)$
is a nonnegative solution of the problem
\begin{equation}
\Delta_{p(x)} u=f \mbox{
in }\Omega,
\end{equation}
with $f\in L^{q_0}(\Omega)$ for some $\max\{1,
\frac{N}{p_-^{4R}}\}<q_0\le\infty$, then
$$\sup_{{B_R}(x_0)}u\leq C[\inf_{B_R(x_0)}u+R+ R\mu]$$
where
$$\mu=[R^{1-\frac{N}{q_0}}||f||_{L^{q_0}(B_{4R}(x_0))}]^{\frac{1}{{p_-^{4R}}-1}}.$$
The constant $C$ depends only on $N$, $p_-^{4R} := \inf_{B_{4R}(x_0)}p$,
$p_+^{4R} := \sup_{B_{4R}(x_0)}p$, $s$, $q_0$, $\omega_{4R}$, ${\mu}^{p_+^{4R}-p_-^{4R}}$,
$||u||_{L^{sq'}(B_{4R}(x_0))}^{p_+^{4R}-p_-^{4R}}$ and $||u||_{L^{sr_0}(B_{4R}(x_0))}^{p_+^{4R}-p_-^{4R}}$ (for certain $q'=\frac{q}{q-1}$ with
$r_0,q\in(1,\infty)$ and $\frac{1}{q_0}+\frac{1}{q}+\frac{1}{r_0}=1$ depending on $N,q_0$ and $p_-^{4R}$). Here
$s>p_+^{4R}-p_-^{4R}$ is arbitrary and $\omega_{4R}$ is the modulus of log-H\"older continuity  of $p(x)$
in $B_{4R}(x_0)$.

\end{theo}

\medskip

We will also use the following result proven in \cite{Fan},
Theorem 1.1,

\begin{theo}\label{regu} Assume that $1<p_{\min}\le p(x)\le p_{\max}<\infty$, and that $p(x)$ has a modulus of continuity $\omega(r)=C_0 r^{\alpha_0}$
for some $0<\alpha_0< 1$. Let $f\in L^{\infty}(\Omega)$ and let
$u\in W^{1,p(\cdot)}(\Omega)\cap L^{\infty}(\Omega)$ be a solution
of
\begin{equation}
\Delta_{p(x)} u=f \mbox{
in }\Omega.
\end{equation}
Then, $u\in C^{1,\alpha}_{\rm loc}(\Omega)$, where the H\"older exponent $\alpha$ depends on $N$, $p_{\min}$, $p_{\max}$,
$||f||_{L^{\infty}(\Omega)}$, $||u||_{L^{\infty}(\Omega)}$, $\omega(r)$ and, for  any $\Omega'\subset \subset\Omega$,
$$\|u\|_{C^{1,\alpha}(\bar\Omega')}\leq C,$$
 the constant $C$ depending  on
$N$, $p_{\min}$, $p_{\max}$, $||f||_{L^{\infty}(\Omega)}$, $||u||_{L^{\infty}(\Omega)}$, $\omega(r)$ and
${\rm dist}(\Omega ', \partial\Omega)$.
\end{theo}

\medskip

In order to prove Theorem \ref{estim-lip}, we need to prove first some
auxiliary results.

\begin{lemm}\label{umenor2ep} Assume that $1<p_{\min}\le p_\ep(x)\le p_{\max}<\infty$
with  $p_\ep(x)$ Lipschitz continuous and $\|\nabla p_\ep\|_{L^{\infty}}\leq L$, for some $L>0$.
Let $u^{\ep}$ be a solution of $P_\ep(\fep, p_\ep)$  in  $B_{r_0}(x_0)$
with $\|u^{\ep}\|_{L^{\infty}(B_{r_0}(x_0))}\leq L_1$,
$\|\fep\|_{L^{\infty}(B_{r_0}(x_0))}\leq L_2$, such that
$u^{\ep}(x_0)\leq 2\ep$. Then, there exists $C>0$
 such
that, if $\ep\leq 1$,
$$|\nabla u^{\ep}(x_0)|\leq C,$$
with $C=C(N,L_1,L_2,\|\beta\|_{L^{\infty}}, p_{\min},p_{\max}, L,r_0)$.
\end{lemm}
\begin{proof}
Let $v^{\ep}(x)=\frac{1}{\ep} u^{\ep}(x_0+\ep x)$. Then, denoting
$\bar p_\ep(x)=p_\ep(\ep x + x_0)$ and  $\bar\fep(x)=\ep\fep(\ep x + x_0)$,
we have, if $\ep\leq1$,
\begin{equation}
\Delta_{\bar p_\ep(x)}v^{\ep}=\beta(v^{\ep})+\bar\fep \mbox{ in
}B_{r_0}.\end{equation}
We will apply Harnack's Inequality (Theorem \ref{har}). Let $\bar r_0=\min\{r_0, 4\}$. We first observe that
$$\gamma:=(\bar p_\ep)_{+}^{\bar r_0}-(\bar p_\ep)_{-}^{\bar r_0}=\sup_{B_{\bar r_0}}\bar p_\ep-\inf_{B_{\bar r_0}}\bar p_\ep\le L\ep2\bar r_0,$$
so that
$$||\vep||_{L^{\infty}(B_{\bar r_0})}^{\gamma}\le ({L_1}/\ep)^{L\ep 2 \bar r_0}\le C_0(L,L_1,r_0).$$
It follows that
$$\sup_{B_{\bar r_0/4}}v^{\ep}\leq C_1[v^{\ep}(0)+\bar r_0/4+\mu \bar r_0/4],$$
for $\mu=\left(\frac{\bar r_0}{4}\|\beta(v^{\ep})+\bar\fep\|_{L^{\infty}(B_{\bar r_0}(x_0))}\right)^{\frac{1}{(\bar p_\ep)_{-}^{\bar r_0} - 1}}
\le C_2(L_2,\|\beta\|_{L^{\infty}}, p_{\min},r_0)$ and a constant $C_1$ with
$C_1=C_1(N,L_1,L_2,\|\beta\|_{L^{\infty}}, p_{\min},p_{\max}, L,r_0)$.

Now, observing that $v^{\ep}(0)\leq 2$, and using
the estimates of Theorem \ref{regu},  we have that
$$|\nabla \uep(x_0)|=|\nabla \vep(0)|\leq C,$$
with $C=C(N,L_1,L_2,\|\beta\|_{L^{\infty}}, p_{\min},p_{\max}, L,r_0).$
\end{proof}

\begin{lemm}\label{cota-barrera}
Assume that $1<p_{\min}\le p(x)\le p_{\max}<\infty$ with $p(x)$ Lipschitz continuous and
$\|\nabla p\|_{L^{\infty}}\leq L$, for some $L>0$. For $x_0\in\Bbb R^{N}$,
$\mu>0$, $\delta>0$, $A>0$, consider
 $$\psi(x)=A\left(\frac{e^{-\mu\frac{|x-x_0|^2}{{\delta}^2}}-e^{-\mu }}{e^{-\mu/16}-e^{-\mu }}\right).$$
 Assume  moreover that $\delta\le A\le A_0$. Then, given $D>0$, there exist
 $\tilde\mu=\tilde\mu(N,p_{\min}, p_{\max})$ and
 $\tilde r= \tilde r(p_{\min}, p_{\max}, L, D, A_0,\mu)$ such that, if $\mu\ge \tilde\mu$ and
 $\delta\le \tilde r$, there holds that
$$\Delta_{p(x)}\psi (x)\ge D\quad \mbox{ in }
B_{\delta}(x_0)\setminus \overline{B_{\delta/4}(x_0)}.$$
\end{lemm}

\begin{proof}
For $M>0$ and  $\mu>0$ let
\begin{equation}\label{defw}
w(x)=M(e^{-\mu|x|^2}-e^{-\mu}).
\end{equation}
The calculations in the proof of Lemma B.4 in \cite{FBMW} show
that if  $q(x)$ is a Lipschitz continuous function, with
$1<p_{\min}\le q(x)\le p_{\max}<\infty$, there exist
${\mu}_0={\mu}_0(p_{\max},p_{\min},N)$ and $\ep_0=\ep_0(p_{\min})$
such that, if $\mu\ge{\mu}_0$ and $\|\nabla q\|_{L^{\infty}}\leq
\ep_0$, then
\begin{align*}&e^{\mu|x|^2}(2M\mu)^{-1}|\nabla w|^{2-q(x)}\Delta_{q(x)}w\geq C_1 \mu-C_2
\|\nabla q\|_{L^{\infty}}(|\log M| +1)\quad
 \mbox{ in }B_1\setminus B_{1/4},\end{align*}
 with $C_1,C_2$ depending only on $p_{\min}$. If, in addition, $\mu\ge\mu_1(p_{\min})$,
 we get
\begin{align*}&e^{\mu|x|^2}(2M\mu)^{-1}|\nabla w|^{2-q(x)}\Delta_{q(x)}w\geq \frac{C_1}{2} \mu-C_2
\|\nabla q\|_{L^{\infty}} |\log M| \quad
 \mbox{ in }B_1\setminus B_{1/4},\end{align*}
 and therefore,
\begin{align*}&\Delta_{q(x)}w\geq e^{-\mu|x|^2}|\nabla w|^{q(x)-2}2M\mu\left(\frac{C_1}{2} \mu-C_2
\|\nabla q\|_{L^{\infty}}|\log M|\right) \quad
 \mbox{ in }B_1\setminus B_{1/4}.\end{align*}
 So that we have
 \begin{align*}&\Delta_{q(x)}w\geq e^{-\mu(p_{\max}-1)}M^{q(x)-1}\mu^{p_{\min}-1}\left(\tilde{C_1} \mu-\tilde{C_2}
\|\nabla q\|_{L^{\infty}} |\log M|\right) \quad
 \mbox{ in }B_1\setminus B_{1/4},\end{align*}
with $\tilde{C_1}$, $\tilde{C_2}$ depending on $p_{\min}$ and
$p_{\max}$  if, in addition, $\mu\ge 1$.

We now observe that, letting in \eqref{defw}
$$M=\frac{A}{\delta(e^{-\mu/16}-e^{-\mu })},$$
we have
 $$\psi(x)=A\left(\frac{e^{-\mu\frac{|x-x_0|^2}{{\delta}^2}}-e^{-\mu }}{e^{-\mu/16}-e^{-\mu }}\right)=
 \delta M\left(e^{-\mu|\frac{x-x_0}{\delta}|^2}-e^{-\mu }\right)=\delta w\left(\frac{x-x_0}{\delta}\right).$$

We want to show that the constants $\tilde\mu$, $\tilde r$ in the statement can be chosen in such a way that
\begin{equation}\label{desig}
\Delta_{p(x)}\psi(x) \ge D\quad \mbox{ in }
B_{\delta}(x_0)\setminus \overline{B_{\delta/4}(x_0)}.
\end{equation}
We notice that  showing \eqref{desig} is equivalent to showing that
\begin{equation}\label{desigdelta}
\Delta_{\bar{p}(x)} w (x)\ge \delta D \quad \mbox{ in }
B_1\setminus \overline{B_{1/4}},
\end{equation}
for $\bar{p}(x)=p(x_0+\delta x)$.

Since $||\nabla \bar{p}||_{L^{\infty}}=\delta||\nabla
p||_{L^{\infty}}\le \delta L$, the previous calculations give, if
$\mu$ is as above and $\delta \le r_1 =\frac{\ep_0}{L}$,
  \begin{align*}&\Delta_{{\bar p(x)}}w\geq e^{-\mu(p_{\max}-1)}M^{\bar p(x)-1}\mu^{p_{\min}-1}\left(\tilde{C_1} \mu-\tilde{C_2}
\delta L |\log M|\right) \quad
 \mbox{ in }B_1\setminus B_{1/4}.\end{align*}

Using that $A\ge \delta$,  we have $M\ge e^{\mu/16}\ge 1$,
implying that

 \begin{align*}\Delta_{{\bar p(x)}}w&\geq e^{-\mu(p_{\max}-1)}
\frac{1}{({{e^{-\mu/16}-e^{-\mu }}})^{p_{\min}-1}}\mu^{p_{\min}-1}\left(\tilde{C_1} \mu-\tilde{C_2}
\delta L \log M\right)\\
&=C_3(\mu)(\tilde{C_1} \mu-\tilde{C_2}
\delta L \log M) \qquad
 \mbox{ in }B_1\setminus B_{1/4}\end{align*}

(here $C_3(\mu)$ is a constant depending on $\mu,
p_{\min},p_{\max}$). Now using that
 \begin{equation*}
 -\delta L \log M\ge -1 -\delta L \mu,
 \end{equation*}
 if $\delta\le r_2=r_2(A_0, L)$ and $\mu\ge\mu_2$,
 we conclude that
  \begin{equation*}\Delta_{{\bar p(x)}}w\geq C_3(\mu)\frac{\tilde{C_1}}{4} \mu \qquad
 \mbox{ in }B_1\setminus B_{1/4},
\end{equation*}
if $\mu\ge\mu_3=\mu_3(p_{\min},p_{\max})$ and $\delta\le
r_3=r_3(p_{\min},p_{\max},L)$. This is,
 \begin{equation*}\Delta_{{\bar p(x)}}w\geq C_5,  \qquad
 \mbox{ in }B_1\setminus B_{1/4}
 \end{equation*}
with $C_5=C_5(\mu,  p_{\min},p_{\max})$. If we now let
$\tilde\mu=\max\{\mu_0, \mu_1, \mu_2, \mu_3,1\}$, fix
$\mu\ge\tilde\mu$ and take $\delta\le \tilde r=\min\{r_1, r_2,
r_3, \frac{C_5}{D}\}$, we conclude that \eqref{desigdelta} holds,
thus implying \eqref{desig}.
\end{proof}

\begin{lemm}\label{uep=ep} Assume that $1<p_{\min}\le p_\ep(x)\le p_{\max}<\infty$
with  $p_\ep(x)$ Lipschitz continuous and $\|\nabla p_\ep\|_{L^{\infty}}\leq L$, for some $L>0$.
Let $u^{\ep}$ be a solution of $P_\ep(\fep, p_\ep)$ in $B_{1}$ with $\|u^{\ep}\|_{L^{\infty}(B_{1})}\leq L_1$,
$\|\fep\|_{L^{\infty}(B_{1})}\leq L_2$
and $0\in\partial\{u^{\ep}>\ep\}$. Then, there exists $0<r_0<1$ such that, for $x\in
B_{r_0}\cap\{u^\ep>\ep\}$ and $\ep\leq 1$,
$$u^{\ep}(x)\leq \ep+C\mbox{dist}(x,\{u^{\ep}\leq \ep\}\cap B_1),$$ with $r_0=r_0(N,L_1,L_2, p_{\min},p_{\max}, L)$ and
$C=C(N,L_1,L_2,\|\beta\|_{L^{\infty}}, p_{\min},p_{\max}, L)$.
\end{lemm}
\begin{proof}
Let $0<r_0<1/4$ be a constant to be chosen later. For $x_0\in B_{r_0}\cap\{u^\ep>\ep\}$, take $m_0=u^{\ep}(x_0)-\ep$
and $\delta_0=\mbox{dist}(x_0,\{u^{\ep}\leq\ep\}\cap B_1)$. Since
$0\in \partial \{u^{\ep}>\ep\}\cap B_1$, $\delta_0\leq r_0$. We
want to prove that $m_0\leq C \delta_0$, with
$C= C(N,L_1,L_2,\|\beta\|_{L^{\infty}}, p_{\min},p_{\max}, L)$.

Since $B_{\delta_0}(x_0)\subset \{u^{\ep}>\ep\}\cap B_1$, we have
that $u^{\ep}-\ep> 0$ in $B_{\delta_0}(x_0)$ and
$\Delta_{p_\ep(x)}(u^{\ep}-\ep)=\fep$. By Harnack's Inequality (Theorem \ref{har})
$$\sup_{B_{\delta_0/4}(x_0)}(\uep-\ep)\leq C_1[\inf_{B_{\delta_0/4}(x_0)}(\uep-\ep)+\delta_0/4+\hat\mu \delta_0/4],$$
for $\hat\mu=\left(\frac{\delta_0}{4}\|\fep\|_{L^{\infty}(B_{\delta_0(x_0)})}\right)^{\frac{1}{(p_\ep)^{\delta_0}_- - 1}}\le C_0(L_2,p_{\min})$, with
$C_1=C_1(N,L_1,L_2, p_{\min},p_{\max}, L)$.
It follows that
$$m_0\leq C_1\inf_{B_{\delta_0/4}(x_0)}(\uep-\ep)+C_2\delta_0,$$
with
$C_2=C_2(N,L_1,L_2, p_{\min},p_{\max}, L)$.

If there holds that $m_0\le 2C_2\delta_0$, the conclusion follows.

So let us assume that $m_0> 2C_2\delta_0$.
Then, there exists
$c_1=c_1(N,L_1,L_2,p_{\min},p_{\max}, L)$ such that
$$c_1 m_0\le\inf_{B_{\delta_0/4}(x_0)}(u^{\ep}-\ep).$$

If $c_1 m_0\le \delta_0$ there is nothing to prove. So now assume that $c_1 m_0> \delta_0$.

Let us consider
 $$\psi(x)=c_1 m_0 \left(\frac{e^{-\mu\frac{|x-x_0|^2}{{\delta_0}^2}}-e^{-\mu }}{e^{-\mu/16}-e^{-\mu }}\right),$$
with $\mu=\tilde\mu(N,p_{\min}, p_{\max})$, the constant in Lemma \ref{cota-barrera}.

Then, observing that $c_1 m_0\le c_1 L_1$, we can apply Lemma
\ref{cota-barrera} with $\delta=\delta_0$, $A=c_1 m_0$,
$A_0=c_1L_1$ and $D=L_2$, if there holds that $\delta_0\le \tilde
r$, where $\tilde r= \tilde r(p_{\min}, p_{\max}, L, D, A_0,\mu)$
is the constant in Lemma \ref{cota-barrera}.

If we choose $r_0=\min\{\tilde r,1/8\}$ above, we have $r_0=r_0(N,L_1,L_2, p_{\min},p_{\max}, L)$ and Lemma \ref{cota-barrera} applies,
so we get
 $$\begin{cases} \Delta_{p_\ep(x)}\psi (x)\ge L_2\ge\fep &\quad \mbox{ in } B_{\delta_0}(x_0)\setminus \overline{B_{\delta_0/4}(x_0)}\\
\psi =0 & \quad \mbox{ on }\partial B_{\delta_0}(x_0)\\
\psi=c_1 m_0 &\quad \mbox{ on }\partial B_{\delta_0/4}(x_0).
\end{cases}
$$

By the comparison principle (see the appendix), we have
\begin{equation}\label{compphiuep} \psi(x)\leq u^{\ep}(x)-\ep \quad \mbox{ in }
\overline{B_{\delta_0}(x_0)}\setminus
{B}_{\delta_0/4}(x_0).\end{equation} Take $y_0\in
\partial {B}_{\delta_0}(x_0)\cap \partial \{u^{\ep}>\ep\}$. Then,
$y_0\in \overline{B_{1/2}}$ and
\begin{equation}\label{psicero}
\psi(y_0)=u^{\ep}(y_0)-\ep=0.
\end{equation}
Let $v^{\ep}(x)=\frac1\ep u^{\ep}(\ep x+y_0)$, ${\bar p}_\ep(x)=p_\ep(\ep x + y_0)$ and  $\bar\fep(x)=\ep\fep(\ep x + y_0)$. Then if $\ep\le 1$
 we have that $\Delta_{{\bar p}_\ep(x)} v^{\ep}=\beta(v^ {\ep})+\bar\fep$ in $B_{1/2}$ and
 $v^{\ep}(0)=1$.
Therefore, by Harnack's Inequality (Theorem \ref{har}),  using similar arguments as those employed in the proof of Lemma \ref{umenor2ep}, we obtain
$\max_{\overline{B}_{1/8}} v^{\ep}\leq \widetilde{c}=\widetilde{c}(N,L_1,L_2,\|\beta\|_{L^{\infty}}, p_{\min},p_{\max}, L)$.

Now, by Theorem \ref{regu}, we get
\begin{equation}\label{graduv14}
|\nabla u^{\ep}(y_0)|=|\nabla v^{\ep}(0)|\leq c_3,
\end{equation}
with $c_3=c_3(N,L_1,L_2,\|\beta\|_{L^{\infty}}, p_{\min},p_{\max},
L)$. Finally, by \eqref{compphiuep}, \eqref{psicero} and
\eqref{graduv14}, we have that $|\nabla \psi(y_0)|\leq |\nabla
u^{\ep}(y_0)|\leq c_3$. Since $|\nabla \psi(y_0)|= c_1 m_0
\frac{c(\mu)}{\delta_0}$, we obtain
$$m_0\leq \frac{c_3}{c_1 c(\mu)}\delta_0$$
and the result follows.

\end{proof}

Now, we can prove the following important result

\begin{prop}\label{cotagradsing}Assume that $1<p_{\min}\le p_\ep(x)\le p_{\max}<\infty$
with  $p_\ep(x)$ Lipschitz continuous and $\|\nabla p_\ep\|_{L^{\infty}}\leq L$, for some $L>0$.
Let $u^{\ep}$ be a solution of $P_\ep(\fep, p_\ep)$ in  $B_1$ with $\|u^{\ep}\|_{L^{\infty}(B_{1})}\leq L_1$ and
$\|\fep\|_{L^{\infty}(B_{1})}\leq L_2$. Assume that  $0\in \partial\{u^{\ep}>\ep\}$.
Then, there exists $0<r_1<1$ such that, for $x\in B_{r_1}$ and $\ep\leq 1$,
$$|\nabla u^{\ep}(x)|\leq C$$ with $r_1=r_1(N,L_1,L_2, p_{\min},p_{\max}, L)$ and $C=C(N,L_1,L_2,\|\beta\|_{L^{\infty}}, p_{\min},p_{\max}, L)$.

\end{prop}

\begin{proof}
By Lemma \ref{umenor2ep} we know that if $x_0\in \{u^{\ep}\leq
2\ep\}\cap B_{3/4}$ then,
$$|\nabla u^{\ep}(x_0)|\leq C_0$$ with
$C_0=C_0(N,L_1,L_2,\|\beta\|_{L^{\infty}}, p_{\min},p_{\max}, L)$.

Let $r_0=r_0(N,L_1,L_2, p_{\min},p_{\max}, L)$ be as in Lemma \ref{uep=ep}.

 Let
$x_0\in B_{r_0/2}\cap \{u^{\ep}>\ep\}$ and
$\delta_0=\mbox{dist}(x_0,\{u^{\ep}\leq \ep\})$.

As $0\in\partial\{u^{\ep}>\ep\}$ we have that $\delta_0\leq r_0/2$.
Therefore, $B_{\delta_0}(x_0)\subset \{u^{\ep}>\ep\}\cap B_{r_0}$
and then $\Delta_{p_\ep(x)} u^{\ep}=\fep$ in $B_{\delta_0}(x_0)$ and,  by Lemma
\ref{uep=ep},
\begin{equation}\label{ecuuepig}
u^{\ep}(x)\leq \ep+ C_1 \mbox{dist}(x,\{u^{\ep}\leq \ep\}) \quad
\mbox{ in } B_{\delta_0}(x_0),
\end{equation}
with $C_1=C_1(N,L_1,L_2,\|\beta\|_{L^{\infty}}, p_{\min},p_{\max}, L)$.
\begin{enumerate}
\item

Suppose that $\ep<\bar{c} \delta_0$ with $\bar{c}$ to be
determined. Then, \eqref{ecuuepig} gives
\begin{equation*}
\sup_{B_{\delta_0}(x_0)} u^{\ep}\leq \ep+ C_1 2\delta_0\le (\bar c+2C_1)\delta_0.
\end{equation*}
Now let $v^{\ep}(x)=\frac1{\delta_0}{u^{\ep}(x_0+\delta_0 x)}$ and $p_\ep^{\delta_0}(x)=p_\ep(x_0+\delta_0 x)$.
Then, we have $\Delta_{p_\ep^{\delta_0}(x)}v^{\ep}=\delta_0{\fep} (x_0+\delta_0  x)$ in $B_1$,
with
$$\sup_{B_1} v^{\ep} = \frac{1}{\delta_0}
\sup_{B_{\delta_0}(x_0)} u^{\ep} \le (\bar c+2C_1).$$

Therefore, by Theorem \ref{regu}
$$|\nabla u^{\ep}(x_0)|=|\nabla v^{\ep}(0)|\leq \widetilde{C},$$
with $\widetilde{C}=\widetilde{C}(N,L_1,L_2,\|\beta\|_{L^{\infty}}, p_{\min},p_{\max}, L, \bar c)$.

\item Suppose that $\ep\geq \bar{c} \delta_0$. By \eqref{ecuuepig}
we have
$$u^{\ep}(x_0)\leq \ep+ C_1 \delta_0\leq \Big (1+\frac{C_1}{\bar{c}}\Big)\ep< 2 \ep,$$
if we choose $\bar{c}$ big enough. By Lemma \ref{umenor2ep}, we
have $|\nabla u^{\ep}(x_0)|\leq C$, with\\
$C=C(N,L_1,L_2,\|\beta\|_{L^{\infty}}, p_{\min},p_{\max}, L)$.
\end{enumerate}
The result follows.
\end{proof}

As a consequence of the previous results we obtain Theorem \ref{estim-lip}. In fact,

\begin{proof}[Proof of Theorem \ref{estim-lip}]
Let $0<\tau<1$ be such that $\forall x\in \Omega'$,
$\overline{B_{2\tau}(x)}\subset \Omega$, and let $\ep\leq \tau$.
Let $r_1$ be the constant in Proposition \ref{cotagradsing}, corresponding to $N$, $\frac{L_1}{\tau}$, $L_2$, $p_{\min},p_{\max}, L$
(i.e., $r_1=r_1(N,\frac{L_1}{\tau},L_2, p_{\min},p_{\max}, L)$).

Let
$x_0\in \Omega'$.

\begin{enumerate}
\item If $\delta_0=\mbox{dist} (x_0,\partial\{u^{\ep}>\ep\})<
\tau r_1$, let $y_0\in\partial\{u^{\ep}>\ep\}$ such that
$|x_0-y_0|=\delta_0$. Let $v^{\ep}(x)=\frac1\tau{u^{\ep}(y_0+\tau x)}$, ${\bar p_\ep}(x)=p_\ep(y_0+\tau x)$, $\bar\fep(x)=\tau\fep(y_0+\tau x)$
and $\bar{x}=\frac{x_0-y_0}{\tau},$ then $|\bar{x}|<r_1$.
There holds that $\|v^{\ep}\|_{L^{\infty}(B_{1})}\leq \frac{L_1}{\tau}$,
 $\|\nabla {\bar p_\ep}\|_{L^{\infty}}\leq L$ and
$\|\bar\fep\|_{L^{\infty}(B_{1})}\leq L_2$.

As $0\in
\partial\{v^{\ep}>\ep/\tau\}$ and $\Delta_{{\bar p_\ep}(x)}v^{\ep}=\beta_{\ep/\tau} (v^{\ep})+\bar\fep$ in
$B_1$, we have by Proposition \ref{cotagradsing}
$$|\nabla u^{\ep}(x_0)|=|\nabla v^{\ep}(\bar{x})|\leq C_1(N,L_1,L_2,\|\beta\|_{L^{\infty}}, p_{\min},p_{\max}, L,\tau).$$
\item If $\delta_0=\mbox{dist} (x_0,\partial\{u^{\ep}>\ep\})\geq
\tau r_1$, there holds that
\begin{enumerate}
\item $B_{\tau r_1}(x_0)\subset \{u^{\ep}>\ep\},$ or \item
$B_{\tau r_1}(x_0)\subset \{u^{\ep}\leq \ep\}.$
\end{enumerate}
In the first case, $\Delta_{p_\ep(x)} u^{\ep}=\fep$ in $B_{\tau r_1}(x_0)$.
Therefore, by Theorem \ref{regu}
$$|\nabla u^{\ep}(x_0)|\leq
C_2(N,L_1,L_2, p_{\min},p_{\max}, L,\tau).$$ In the second case, we can apply
Lemma \ref{umenor2ep} and we have,
$$|\nabla u^{\ep}(x_0)|\leq
C_3(N,L_1,L_2,\|\beta\|_{L^{\infty}}, p_{\min},p_{\max}, L,\tau).$$ The result is
proved.
\end{enumerate}
\end{proof}

\end{section}

\begin{section}{Passage to the limit}

Since we have that $|\nabla u^{\ep}|$ is locally bounded by a
constant independent of $\ep$, we have that there exists a
function $u \in Lip_{\rm{loc}}(\Omega)$ such that, for a
subsequence $\ep_j\to 0$, $u^{\ep_j}\to u$. In this section we
will prove some properties of the function $u$.

\begin{lemm}
\label{clw1-Lemma 3.1} Let $\uep$ be a family of solutions to
\begin{equation}
\tag{$P_\ep(\fep, p_\ep)$}
\Delta_{p_\ep(x)}\uep={\beta}_{\varepsilon}(\uep)+\fep, \quad
u^{\ep}\geq 0
\end{equation}
in a domain $\Omega\subset \Bbb R^{N}$. Let us assume that
$||\uep||_{L^\infty(\Omega)}\le L_1$ and
$||\fep||_{L^\infty(\Omega)}\le L_2$ for some $L_1>0$, $L_2
>0$. Assume moreover that $1<p_{\min}\le p_\ep(x)\le p_{\max}<\infty$ and that
$p_\ep(x)$ are Lipschitz continuous with $\|\nabla
p_\ep\|_{L^{\infty}}\leq L$, for some $L>0$.

Then, for any sequence $\ep_{j}\to 0$ there exist a subsequence
$\ep'_{j}\to 0$ and functions $u\in
\mbox{Lip}_{\rm{loc}}(\Omega)$, $f\in L^\infty(\Omega)$ and $p\in
\mbox{Lip}(\Omega)$, with $1<p_{\min}\le p(x)\le p_{\max}<\infty$
and $\|\nabla p\|_{L^{\infty}}\leq L$, such that
\begin{enumerate}
\item $u^{\ep'_j}\to u$ uniformly on compact subsets of $\Omega$,
\item $f^{\ep'_j}\rightharpoonup f$   $*-$weakly in $L^\infty(\Omega)$,
\item $p_{\ep'_j}\to p$ uniformly on compact subsets of $\Omega$,
\item $\Delta_{p(x)} u  \ge f$ in the distributional sense in
$\Omega$,
\item $\Delta_{p(x)} u = f$ in  $\{u>0\}$.
\item There exists a nonnegative Radon  measure $\mu$ such that
$\beta_{\ep'_j}(u^{\ep'_j})\rightharpoonup \mu$ as measures in
$\Omega'$, for every $\Omega'\subset\subset \Omega$.
\item There holds
$$
-\int_{\Omega} |\nabla u|^{p(x)-2}\nabla u \cdot \nabla
\varphi\, dx =\int_{\Omega} \varphi \, d\mu + \int_{\Omega} f \varphi \, dx
$$
for every $\varphi\in C^{\infty}_0(\Omega)$. \item $\nabla
u^{\ep'_j}\rightharpoonup \nabla u$ weakly in
$L_{\rm{loc}}^{p(\cdot)}(\Omega)$. \item If  $p(x)\equiv p_0$,
with $p_0$ a constant, then $\nabla u^{\ep'_j}\rightarrow \nabla
u$ in $L_{\rm{loc}}^{p_0}(\Omega)$.
\end{enumerate}
\end{lemm}

\medskip

\begin{proof}
(1) and (8) follow by Theorem \ref{estim-lip}. (2) and (3) are immediate.

In order to prove (5), take $E\subset\subset E'\subset\subset
\{u>0\}$. Then, $u\geq c>0$ in $E'$. Therefore, $u^{\ep'_j}>c/2$
in $E'$ for $\ep_j'$ small. If we take $\ep_j'<c/2$ --as
$\Delta_{p_{\ep_j'}(x)}u^{\ep'_j} = f^{\ep_j'}$ in $\{u^{\ep_j'}>\ep_j'\}$-- we have
that $\Delta_{p_{\ep_j'}(x)}u^{\ep'_j} = f^{\ep_j'}$ in $E'$. Therefore, by Theorem \ref{regu}, $\|u^{\ep_j'}\|_{C^{1,\alpha}(\bar E)}\leq C$.

Thus, for a subsequence, we have
$$\nabla u^{\ep_j'}\rightarrow \nabla u \quad \mbox{ uniformly  in } E.$$

Therefore, $\Delta_{p(x)} u = f$ in $E$.

In order to prove (6), let us take $\Omega'\subset\subset \Omega$,
and $\varphi\in C_0^{\infty}(\Omega)$, $\varphi\ge 0$, with $\varphi=1$ in
$\Omega'$ as a test function in $ P_{\ep_j}(f^{\ep_j}, p_{\ep_j})$. Since  $\|\nabla
u^{\ep_j'}\| \leq C$ in $\Omega'$, there holds that
\begin{equation}\label{estimL1bep}
C(\varphi)\geq \int_{\Omega} \beta_{\ep_j'}(u^{\ep_j'})\varphi\, dx
\geq \int_{\Omega'} \beta_{\ep_j'}(u^{\ep_j'})\, dx.
\end{equation}
Therefore, $ \beta_{\ep_j'}(u^{\ep_j'})$ is bounded in
$L^1_{\rm{loc}}(\Omega)$, so that, there exists a locally finite
measure $\mu$ such that
$$\beta_{\ep_j'}(u^{\ep_j'})\rightharpoonup\mu \quad \mbox{ as
measures. }$$
That is, for every $\varphi \in C_0(\Omega)$,
$$\int_{\Omega}\beta_{\ep_j'}(u^{\ep_j'})\varphi\, dx \rightarrow \int_{\Omega} \varphi\, d\mu.$$

We will divide the reminder of the proof into several steps.

Let $\Omega'\subset\subset \Omega$. We will show that for every
$v\in C_0^{\infty}(\Omega')$ there holds that
\begin{equation}\label{eqlim}
\int_{{\Omega}'} |\nabla u^{\ep_j'}|^{p_{\ep_j'}(x)-2}\nabla u^{\ep_j'} \cdot \nabla
v\, dx \to \int_{{\Omega}'} |\nabla u|^{p(x)-2}\nabla u \cdot \nabla
v\, dx.
\end{equation}

Let us denote, for $\eta\in \R^N$, $A^{\epj}(x, \eta)=|\eta|^{p_{\ep_j}(x)-2}\eta$ and
$A(x, \eta)=|\eta|^{p(x)-2}\eta$.

By Theorem \ref{estim-lip}, we have $|\nabla u ^{\ep_j}|\leq C$ in $\Omega'$.
Therefore for a subsequence $\ep'_j$ we have that there exists
$\xi\in (L^{\infty}(\Omega'))^N$ such that,
\begin{equation}\label{limitelim}
\begin{aligned} &\nabla u^{\ep_j'}
\rightharpoonup \nabla u \quad  &*-\mbox{weakly in }
(L^\infty(\Omega'))^N\\
  &A^{{\ep_j}'}(x, \nabla u^{\ep_j'})
\rightharpoonup\xi \quad  &*-\mbox{weakly in }
(L^\infty(\Omega'))^N\\
&u^{\ep'_j}\to u &\mbox{ uniformly in } \Omega'.
\end{aligned}\end{equation}

For simplicity we call $\ep_j'=\ep$, $A^{\epj}(x, \eta)=A^\ep(\eta)$ and $A(x, \eta)=A(\eta)$.

Step 1. Let us  prove that for any $v\in C(\Omega')\cap
W^{1,\infty}(\Omega')$ there holds that
\begin{equation}\label{primero}
\int_{\Omega'}(\xi-A(\nabla u))\nabla v\,dx=0.
\end{equation}

In fact, as  $A^\ep$ is monotone (i.e $
\big(A^\ep(\eta)-A^\ep(\zeta)\big)\cdot(\eta-\zeta)\geq 0\  \forall
\eta,\zeta\in\R^N$)
 we have that, for any $w\in W^{1,\infty}(\Omega')$,
\begin{equation}\label{monotono1}
I=\int_{\Omega'} \big(A^\ep(\nabla u^{\ep})-A^\ep(\nabla w)\big) (\nabla
u^{\ep}-\nabla w)\, dx \geq 0.
\end{equation}
Therefore, if $\psi\in C_0^\infty(\Omega')$,
\begin{equation}\label{monotono2}\begin{aligned} &-\int_{\Omega'}
\beta_{\ep}(u^{\ep}) u^{\ep}\, dx -\int_{\Omega'}A^\ep(\nabla
u^{\ep}) \nabla w\, dx -\int_{\Omega'} A^\ep(\nabla w) (\nabla
u^{\ep}-\nabla w)\, dx \\&=-\int_{\Omega'} \beta_{\ep}(u^{\ep})
u^{\ep}\, dx -\int_{\Omega'}
A^\ep(\nabla u^{\ep}) \nabla u^{\ep}\, dx +I\\
&=-\int_{\Omega'} \beta_{\ep}(u^{\ep}) u\, dx-\int_{\Omega'}
\beta_{\ep}(u^{\ep}) (u^{\ep}-u)\psi\, dx -\int_{\Omega'}
\beta_{\ep}(u^{\ep}) (u^{\ep}-u) (1-\psi)\, dx \\&\ \ \
-\int_{\Omega'}
A^\ep(\nabla u^{\ep}) \nabla u^{\ep}\, dx+I \\
&\geq-\int_{\Omega'} \beta_{\ep}(u^{\ep}) u\, dx+\int_{\Omega'}
A^\ep(\nabla u^{\ep}) \nabla(u^{\ep}-u) \psi\, dx+\int_{\Omega'}
A^\ep(\nabla u^{\ep}) (u^{\ep}-u)\nabla \psi\, dx\\& \ \  \
-\int_{\Omega'} \beta_{\ep}(u^{\ep}) (u^{\ep}-u) (1-\psi)\, dx
 -\int_{\Omega'}
A^\ep(\nabla u^{\ep}) \nabla u^{\ep}\, dx + \int_{\Omega'} \fep (u^{\ep}-u)\psi\, dx,
\end{aligned}\end{equation}
where in the last inequality we are using \eqref{monotono1} and
equation $P_\ep(\fep, p_\ep)$.

Now, take  $\psi=\psi_j\to \chi_{\Omega'}$\, a.e., with $0\le\psi_j\le 1$. If $\Omega'$ is smooth
we can choose the functions so that $\int |\nabla \psi_j|\, dx \le C \mbox{Per
}\Omega'$. Therefore,
$$
\Big|\int_{\Omega'}A^\ep(\nabla u^{\ep}) (u^{\ep}-u)\nabla \psi_j\, dx
\Big|\leq C\|u^{\ep}-u\|_{L^{\infty}(\Omega')} \int_{\Omega'}
|\nabla \psi_j|\, dx\leq C \|u^{\ep}-u\|_{L^{\infty}(\Omega')}.
$$
Also
$$
\Big|\int_{\Omega'}\fep (u^{\ep}-u)\psi_j\, dx\Big|
\leq C \|u^{\ep}-u\|_{L^{\infty}(\Omega')},
$$
and
$$
\Big|\int_{\Omega'} \beta_{\ep}(u^{\ep}) (u^{\ep}-u)\, dx\Big| \leq C \|u^{\ep}-u\|_{L^{\infty}(\Omega')}.
$$
So that, with this choice of $\psi=\psi_j$ in \eqref{monotono2}, we
obtain

\begin{align*}
&-\int_{\Omega'} \beta_{\ep}(u^{\ep}) u^{\ep}\, dx -\int_{\Omega'}
A^\ep(\nabla u^{\ep}) \nabla w\, dx -\int_{\Omega'} A^\ep(\nabla w)
(\nabla u^{\ep}-\nabla w)\, dx  \\&\geq -\int_{\Omega'}
\beta_{\ep}(u^{\ep}) u\, dx+\int_{\Omega'} A^\ep(\nabla u^{\ep})
\nabla(u^{\ep}-u) \, dx- C \|u^{\ep}-u\|_{L^{\infty}(\Omega')}
-\int_{\Omega'} A^\ep(\nabla u^{\ep}) \nabla u^{\ep}\, dx
\\&=-\int_{\Omega'}
\beta_{\ep}(u^{\ep}) u\, dx-\int_{\Omega'} A^\ep(\nabla u^{\ep})
\nabla u \, dx- C \|u^{\ep}-u\|_{L^{\infty}(\Omega')}
\\&\ge-\int_{\Omega'}
\beta_{\ep}(u^{\ep}) u^{\ep}\, dx-\int_{\Omega'} A^\ep(\nabla u^{\ep})
\nabla u \, dx- C \|u^{\ep}-u\|_{L^{\infty}(\Omega')}
.
\end{align*}

 Therefore, canceling $\int_{\Omega'}\beta_{\ep}(u^{\ep}) u^{\ep}\, dx$ first, and then, letting $\ep\to 0$ we get by using \eqref{limitelim} and (3) that
\begin{align*} &
-\int_{\Omega'} \xi \nabla w\, dx -\int_{\Omega'} A(\nabla w)
(\nabla u-\nabla w)\, dx \geq
-\int_{\Omega'} \xi \nabla u\, dx
\end{align*}
and then,
\begin{equation}\label{ok}
 \int_{\Omega'} (\xi-A(\nabla w)) (\nabla u-\nabla w)\,
dx \geq 0.
\end{equation}
Take now $w=u-\lambda v$ with $v\in C(\Omega')\cap
W^{1,\infty}(\Omega')$ and $\lambda>0$. Dividing by $\lambda$ and
taking $\lambda\to 0^+$ in \eqref{ok}, we obtain
$$\int_{\Omega'} (\xi-A(\nabla u)) \nabla v\, dx
\geq 0.
$$
Replacing $v$ by $-v$ we obtain \eqref{primero}. Then, \eqref{eqlim} holds which implies (7) and (4).

In order to prove (9) let us now assume that $p(x)\equiv p_0$, with $p_0$ a constant.
Then  we now have $A(x,\eta)=A(\eta)=|\eta|^{p_0-2}\eta$.

 Step 2. Let us prove that
 \begin{equation}\label{step11}
 \int_{\Omega'}A^{\ep}(\nabla\uep)\nabla \uep\to\int_{\Omega'}A(\nabla u)\nabla u.
 \end{equation}

 \medskip

By passing to the limit in the equation
\begin{equation}\label{Auep}
0=\int_{\Omega'} A^{\ep}(\nabla u^{\ep}) \nabla
\phi+\int_{\Omega'} \beta_{\ep}(u^\ep)  \phi+ \int_{\Omega'} \fep
\phi \, dx,
\end{equation}
 we have, by Step 1, that for every
 $\phi\in C_0(\Omega')\cap W^{1,\infty}(\Omega')$,
\begin{equation}\label{Auu}
 0=\int_{\Omega'}
A(\nabla u) \nabla \phi+\int_{\Omega'}  \phi\, d\mu+ \int_{\Omega'} f\phi \, dx.
\end{equation}

 On the other hand, taking $\phi=u^{\ep} \psi$ in \eqref{Auep} with $\psi \in C_0^{\infty}(\Omega')$ we have that
$$0=\int_{\Omega'} A^{\ep}(\nabla u^{\ep}) \nabla u^{\ep}
\psi\, dx + \int_{\Omega'} A^{\ep}(\nabla u^{\ep}) u^{\ep} \nabla\psi\,
dx+\int_{\Omega'} \beta_{\ep}(u^{\ep})  u^{\ep}\psi\, dx +
\int_{\Omega'} \fep u^{\ep}\psi\, dx.$$
Using that $A^{\ep}(\nabla u^{\ep}) u^{\ep} \nabla\psi\, \to A(\nabla u) u \nabla\psi$ a.e. in $\Omega'$, with
$|A^{\ep}(\nabla u^{\ep}) u^{\ep} \nabla\psi|\le C$ in $\Omega'$, we get
\begin{align*}
\int_{\Omega'} A^{\ep}(\nabla u^{\ep}) u^{\ep} \nabla\psi\, dx & \to
\int_{\Omega'}
A(\nabla u) u \nabla\psi\, dx \\
 \int_{\Omega'} \beta_{\ep}( u^{\ep}) u^{\ep}
\psi\, dx &\to \int_{\Omega'}  u \psi d\mu.
\end{align*}
Then  we obtain $$0=\lim_{\ep\to 0} \Big(\int_{\Omega'} A^{\ep}(\nabla
u^{\ep}) \nabla u^{\ep} \psi\, dx \Big)+ \int_{\Omega'} A(\nabla
u) u \nabla\psi\, dx+
 \int_{\Omega'} u \psi d\mu+\int_{\Omega'} f u\psi\, dx. $$
Now taking, $\phi=u \psi$ in \eqref{Auu} we have $$
0=\int_{\Omega'} A(\nabla u) \nabla u \psi\, dx + \int_{\Omega'}
A(\nabla u) u \nabla\psi\, dx+
 \int_{\Omega'} u \psi\, d\mu+\int_{\Omega'} f u\psi\, dx.
$$ Therefore, $$\lim_{\ep\to 0} \int_{\Omega'} A^{\ep}(\nabla u^{\ep})
\nabla u^{\ep} \psi\, dx= \int_{\Omega'} A(\nabla u) \nabla u
\psi\, dx. $$ Then,
\begin{align*}&\left|\int_{\Omega'}
(A^{\ep}(\nabla u^{\ep}) \nabla u^{\ep}-A(\nabla u) \nabla u )\,
dx\right|\\& \leq \left|\int_{\Omega'} (A^{\ep}(\nabla u^{\ep}) \nabla
u^{\ep}-A(\nabla u) \nabla u )\psi\, dx\right|+
\left|\int_{\Omega'} (A^{\ep}(\nabla u^{\ep}) \nabla u^{\ep} )(1-\psi)\,
dx\right|\\&\ \ +\left|\int_{\Omega'} A(\nabla u) \nabla u
(1-\psi)\, dx\right|
\\& \leq \left|\int_{\Omega'} (A^{\ep}(\nabla u^{\ep}) \nabla
u^{\ep}-A(\nabla u) \nabla u )\psi\, dx\right|+ C\int_{\Omega'}
|1-\psi|\, dx
\end{align*}
so that taking $\ep\to 0$ and then $\psi\to 1$ a.e. with
$0\leq\psi\leq 1$ we obtain \eqref{step11}. This is,
\begin{equation}\label{step11b}
\int_{\Omega'} |\nabla u^{\ep}|^{p_{\ep}(x)}dx\rightarrow
\int_{\Omega'} |\nabla u|^{p_0}dx.
\end{equation}

Step 3.  Let us prove that
\begin{equation}\label{Guepau}\int_{\Omega'} |\nabla
u^{\ep}|^{p_0}\, dx \rightarrow \int_{\Omega'} |\nabla u|^{p_0}\,
dx.\end{equation}

We first observe that

\begin{equation}\label{acerco}\Big|\int_{\Omega'} |\nabla
u^{\ep}|^{p_{\ep}(x)}\, dx -\int_{\Omega'} |\nabla \uep|^{p_0}\,
dx\Big|\le
\int_{\Omega'} \big| |\nabla
u^{\ep}|^{p_{\ep}(x)} - |\nabla \uep|^{p_0}\big|\,
dx \rightarrow 0.\end{equation}
Here we have used that $\big| |\nabla
u^{\ep}|^{p_{\ep}(x)} - |\nabla \uep|^{p_0}\big|\to 0$ a.e. in $\Omega'$ with
$\big| |\nabla u^{\ep}|^{p_{\ep}(x)} - |\nabla \uep|^{p_0}\big|\le C$ in $\Omega'$.

Thus,   \eqref{Guepau} follows from  \eqref{step11b} and \eqref{acerco}.

 Step 4. End of the proof of (9).

Since $\uep\rightharpoonup u$ weakly in
$W^{1,p_0}_{\rm{loc}}(\Omega)$ and
$||\uep||_{W^{1,p_0}(\Omega')}\to ||u||_{W^{1,p_0}(\Omega')}$, for
every $\Omega'\subset\subset\Omega$, it follows that $\uep\to u$
in $W^{1,p_0}_{\rm{loc}}(\Omega)$. In particular,
$\nabla\uep\to\nabla u$ in $L^{p_0}_{\rm{loc}}(\Omega)$. This
completes the proof of the lemma.
 \end{proof}

\medskip

\begin{lemm}\label{radpos}
Let $v$ be a continuous nonnegative function in a domain
$\Omega\subset\Bbb R^{N}$, $v\in W^{1,p(\cdot)}(\Omega)$, such that $\Delta_{p(x)} v = g$ in
$\{v>0\}$ with $g\in L^\infty(\Omega)$. Then $\lambda_v:=
\Delta_{p(x)} v - g \chi_{\{v
>0\}}$ is a nonnegative Radon measure with support on
$\Omega\cap\partial\{v>0\}$.
\end{lemm}
\begin{proof} The proof follows as in the case $p(x)\equiv 2$, that was done in \cite{LW3}, Lemma 2.1.
\end{proof}

\medskip

\begin{coro}
\label{nuevaclw1-Proposition 3.1}
 Let $\uepj$ be a family of  solutions to $\pepj(\fepj, p_{\epj})$ in a domain $\Omega\subset
\Bbb R^{N}$  with $1<p_{\min}\le p_{\epj}(x)\le p_{\max}<\infty$ and
$p_{\epj}(x)$  Lipschitz continuous with $\|\nabla
p_{\epj}\|_{L^{\infty}}\leq L$, for some $L>0$.
Assume that $\uepj\to u$ uniformly on compact subsets of
$\Omega$, $\fepj\rightharpoonup f$ $*-$weakly in $L^\infty(\Omega)$, $p_{\epj}\to p$ uniformly on compact subsets of $\Omega$ and
$\ep_j\to 0$. Then,
\begin{equation*}
\Delta_{p(x)} u - f \chi_{\{u>0\}}=\lambda_u \quad\mbox{ in
}\Omega,
\end{equation*}
with $\lambda_u$ a nonnegative Radon measure supported  on the
free boundary $\Gamma=\Omega\cap\partial\{u>0\}$.
\end{coro}
\begin{proof} It is an immediate consequence of Lemma \ref{clw1-Lemma 3.1} and Lemma \ref{radpos}.
\end{proof}

\medskip

\begin{lemm}
\label{clw1-Lemma 3.2}
 Let $\uepj$ be a family of  solutions to $\pepj(\fepj, p_{\epj})$ in a domain $\Omega\subset
\Bbb R^{N}$ with $1<p_{\min}\le p_{\epj}(x)\le p_{\max}<\infty$ and
$p_{\epj}(x)$  Lipschitz continuous with $\|\nabla
p_{\epj}\|_{L^{\infty}}\leq L$, for some $L>0$.
Assume that $\uepj\to u$ uniformly on compact subsets of
$\Omega$, $\fepj\rightharpoonup f$ $*-$weakly in $L^\infty(\Omega)$, $p_{\epj}\to p$ uniformly on compact subsets of $\Omega$ and
$\ep_j\to 0$.

Let  $x_0\in \Omega$ and $x_n\in\Omega$  be such that $u(x_0)=0$, $u(x_n)=0$ and $x_n\to x_0$ as $n\to\infty$.
Let $\Ln\to 0$, $u_{\Ln}(x)=\frac1{\Ln} u(x_n+\Ln x)$ and
$(\uepj)_{\Ln}(x)=\frac1{\Ln} \uepj(x_n+\Ln x)$.
 Assume that   $u_{\lambda_n}\to U$ as $n\to \infty$
uniformly on compact sets of $\Bbb R^{N}$. Then, there exists
$j(n)\to+\infty$ such that for every $j_n\ge j(n)$ there holds
that $\frac\epjn{\lambda_n}\to 0$ and
\medskip
\item {1)} $\uepjnln\to U$ uniformly on compact sets of $\Bbb
R^{N}$,
\smallskip
\item{2)} $\nabla\uepjnln\to\nabla U\mbox{ in }L^{p_0}_{\rm{loc}}(\Bbb
R^{N})$ with $p_0=p(x_0)$.
\end{lemm}
\begin{proof} The result follows from Lemma \ref{clw1-Lemma 3.1} exactly as Lemma 3.2 in \cite{CLW1}.
\end{proof}

\end{section}

\begin{section}{Basic Limits}

In this section we analyze some  limits that are crucial in the understanding of general limits.

\medskip

We start with the following lemma

\begin{lemm}\label{noweiss} Let $\uepj$, $\fepj$, $p_{\epj}$, $\ep_j$, $u$, $f$ and $p$ be as in Lemma \ref{clw1-Lemma 3.2}.

 Then there exists $\chi\in
L^1_{\rm{loc}}(\Omega)$ such that, for a subsequence,
$\bepj(\uepj)\to\chi$ in $L^1_{\rm{loc}}(\Omega)$, with
$\chi\equiv M$ in $\{u>0\}$ and
$\chi(x)\in\{0,M\}$ a.e. in $\Omega$. If, in addition,
$\fepj\rightharpoonup 0$ in $\{u\equiv 0\}^\circ$, there holds that  $\chi\equiv M$ or
$\chi\equiv 0$ on every connected component of $\{u\equiv 0\}^\circ$.
\end{lemm}
\begin{proof} We first observe that, for every $K\subset\subset\Omega$, there holds
\begin{equation}
\label{estgradBe}
 \int_{K}|\nabla
B_{{\epj}}(u^{\epj})|=\int_{K}\beta_{{\epj}}(u^{\epj})|\nabla
u^{\epj}| \le C_K\int_{K} \beta_{{\epj}}(u^{\epj}),
\end{equation}
where the last term is bounded by a constant $C'_K$ due to
estimate (\ref{estimL1bep}).

Since $0\le B_{{\epj}}(u^{\epj})\le M$, then, there exists
$\chi\in L^1_{\rm{loc}}(\Omega)$ such that, for a subsequence,
$\bepj(\uepj)\to\chi$ in $L^1_{\rm{loc}}(\Omega)$.

Proceeding as in the case $p(x)\equiv 2$ (see \cite{LW3}, Lemma 3.1) we deduce that
$\chi\equiv M$ in $\{u>0\}$ and $\chi(x)\in\{0,M\}$ a.e. in $\Omega$.

Finally, if $\fepj\rightharpoonup 0$ in $\{u\equiv 0\}^\circ$, we take
$K\subset\subset\{u\equiv 0\}^\circ$ in (\ref{estgradBe}) and we
observe that the last term there goes to zero since, by (6) and (7) in Lemma \ref{clw1-Lemma 3.1},
$\beta_{\epj}(u^{\epj})\rightharpoonup \mu$ locally as measures, with $\mu=0$ in $K$.
Thus the result follows.
\end{proof}

\medskip

\begin{prop}
\label{clw1-Proposition 5.2} Let $\uepj$ be solutions to $\pepj(\fepj, p_{\epj})$ in a domain $\Omega\subset
\Bbb R^{N}$ with $1<p_{\min}\le p_{\epj}(x)\le p_{\max}<\infty$ and
$p_{\epj}(x)$  Lipschitz continuous with $\|\nabla p_{\epj}\|_{L^{\infty}}\to 0$.
Let $x_0\in \Omega$ and suppose $\uepj$ converge  to
$u_0=\alpha(x-x_0)_1^+$ uniformly on compact subsets of  $\Omega$,
with $\alpha\in\Bbb R $, $\fepj\rightharpoonup0$ $*-$weakly in
$L^\infty(\Omega)$, $p_{\epj}\to p_0$ uniformly on compact subsets of
$\Omega$, with $p_0\in\Bbb R$,
 and  $\ep_j\to 0$.
Then
$$
\alpha =0\quad\mbox{or} \quad
\alpha=\Big(\frac{p_0}{p_0-1}\,M\Big)^{1/p_0},
$$
with $\int \beta(s)\, ds=M$.
\end{prop}
\begin{proof}
Assume, for simplicity, that $x_0=0$. Since $u^{\epj}\geq 0$, we
have that $\alpha\geq 0$. If $\alpha=0$ there is nothing to prove.
So let us assume that $\alpha>0$.

Let $\psi \in C_0^{\infty}(\Omega)$. We claim that there holds
that
\begin{equation}\label{novale}
\begin{aligned}
-\int_{\Omega}&\frac{|\nabla u^{\epj}|^{p_{\epj}}}{p_{\epj}} \psi_{x_1}\, dx +
\int_{\Omega} |\nabla u^{\epj}|^{p_{\epj}-2} \nabla u^{\epj}\cdot \nabla \psi
\, u^{\epj}_{x_1}\, dx +\int_{\Omega} f^{\epj}u^{\epj}_{x_1}\psi\, dx=\\
& \int_{\Omega}\frac{|\nabla u^{\epj}|^{p_{\epj}}}{p_{\epj}}\log|\nabla u^{\epj}| (p_{\epj})_{x_1}\psi\, dx
-\int_{\Omega}\frac{|\nabla u^{\epj}|^{p_{\epj}}}{p_{\epj}^2} (p_{\epj})_{x_1}\psi\, dx+\int_{\Omega} B_{\epj}(u^{\epj})\psi_{x_1}\, dx.
\end{aligned}
\end{equation}

In fact, let $\Omega'\subset\subset\Omega$ be smooth and let $v_n$
be such that
\begin{equation}
\begin{cases}\label{ecua-n}
\mbox{div}((\frac{1}{n}+|\nabla
v_n|^2)^{\frac{{p_{\ep}}(x)-2}{2}}\nabla v_n)
={\beta}_{{\varepsilon}}(u^{\ep})+f^{\ep}=g^{\ep} &\quad \mbox{ in } \Omega'\\
v_n=u^{\ep} &\quad \mbox{ on }\partial \Omega',
\end{cases}
\end{equation}
were for simplicity we have denoted $\varepsilon_j=\varepsilon$.
By the results in \cite{Fan} and \cite{CL}, $v_n\in
C^{1,\alpha}(\overline \Omega')\cap W^{2,2}_{\rm{loc}} (\Omega')$,
with $||v_n||_{C^{1,\alpha}(\overline \Omega')}\le C$, with $C$
independent of $n$, and therefore, there exists $v_0$ such that,
for a subsequence,
\begin{align*}&v_n\rightarrow v_0 \quad \mbox{uniformly in } \Omega'\\
&\nabla v_n\rightarrow \nabla v_0 \quad \mbox{uniformly in
}\Omega'.
\end{align*}
We get $\Delta_{p_{\ep}(x)}v_0=\Delta_{p_{\ep}(x)}\uep=g^{\ep}$ in
$\Omega'$, with $v_0=\uep$ in $\partial \Omega'$ and therefore,
$v_0=\uep$.

In order to get \eqref{novale} we take as test function in the
weak formulation of \eqref{ecua-n} the function $\psi
{v_n}_{x_1}$, with $\psi \in C_0^{\infty}(\Omega')$. It follows
that

\begin{equation}\label{alfa}
\begin{aligned}
-\int_{\Omega}\Big(\frac{1}{n} + |\nabla {v_n}|^2\Big)^{\frac{{p_{\ep}-2}}{2}} &\nabla {v_n}\cdot \nabla {v_n}_{x_1}\, \psi
 \, dx=\\
&\int_{\Omega} \Big(\frac{1}{n} + |\nabla {v_n}|^2\Big)^{\frac{{p_{\ep}-2}}{2}} \nabla {v_n}\cdot \nabla \psi
\, {v_n}_{x_1}\, dx +\int_{\Omega} g^{\ep}{v_n}_{x_1}\psi\, dx.
\end{aligned}
\end{equation}
On the other hand,
\begin{equation}\label{beta}
\begin{aligned}
-\int_{\Omega}\frac{(\frac{1}{n} + |\nabla {v_n}|^2)^{\frac{{p_{\ep}}}{2}}}{p_{\ep}} \,{\psi}_{x_1}
 \, dx
 =&\int_{\Omega}\frac{{(\frac{1}{n} + |\nabla {v_n}|^2)}^{\frac{{p_{\ep}}}{2}}}{p_{\ep}}\,
 {\frac12}\log{\Big(\frac{1}{n} + |\nabla {v_n}|^2\Big)}{(p_{\ep})}_{x_1}\psi
 \, dx\\
 -\int_{\Omega}\frac{(\frac{1}{n} + |\nabla {v_n}|^2)^{\frac{{p_{\ep}}}{2}}}{p_{\ep}^2} &\,{(p_{\ep})}_{x_1}\psi
 \, dx +
\int_{\Omega}\Big(\frac{1}{n} + |\nabla {v_n}|^2\Big)^{\frac{{p_{\ep}-2}}{2}} \nabla {v_n}\cdot \nabla {v_n}_{x_1}\, \psi
 \, dx.
 \end{aligned}
\end{equation}
Then,  recalling that $g^{\ep}={\beta}_{{\varepsilon}}(u^{\ep})+f^{\ep}$,  we obtain from \eqref{alfa} and \eqref{beta}
\begin{equation*}
\begin{aligned}
&-\int_{\Omega}\frac{(\frac{1}{n} + |\nabla {v_n}|^2)^{\frac{{p_{\ep}}}{2}}}{p_{\ep}} \,{\psi}_{x_1}
 \, dx +\int_{\Omega} \Big(\frac{1}{n} + |\nabla {v_n}|^2\Big)^{\frac{{p_{\ep}-2}}{2}} \nabla {v_n}\cdot \nabla \psi
\, {v_n}_{x_1}\, dx +\int_{\Omega} f^{\ep}{v_n}_{x_1}\psi\, dx
 =\\
& \int_{\Omega}\frac{{(\frac{1}{n} + |\nabla {v_n}|^2)}^{\frac{{p_{\ep}}}{2}}}{p_{\ep}}\,
 \log{\Big(\frac{1}{n} + |\nabla {v_n}|^2\Big)}^{\frac12}{p_{\ep}}_{x_1}\psi
 \, dx
 -\int_{\Omega}\frac{(\frac{1}{n} + |\nabla {v_n}|^2)^{\frac{{p_{\ep}}}{2}}}{p_{\ep}^2} \,{p_{\ep}}_{x_1}\psi
 \, dx
  -\int_{\Omega} {\beta}_{{\varepsilon}}(u^{\ep}){v_n}_{x_1}\psi\, dx.
 \end{aligned}
\end{equation*}
Passing to the limit as $n\to \infty$ and integrating by parts in the last term, we get \eqref{novale}.

Now, by Lemma
\ref{noweiss}, we have that there exists
  $\chi\in L^1_{\rm{loc}}(\Omega)$ such that, for a subsequence,
$\bepj(\uepj)\to\chi$ in $L^1_{\rm{loc}}(\Omega)$. This, together with
 the strong convergence result in Lemma \ref{clw1-Lemma 3.1} and the fact that $\|\nabla p_{\epj}\|_{L^{\infty}}\to 0$ gives, when passing to the limit in \eqref{novale},
\begin{equation}\label{novale2}
-\int_{\Omega}\frac{|\nabla u_0|^{p_0}}{{p_0}} \psi_{x_1}\, dx +
\int_{\Omega} |\nabla u_0|^{p_0-2} \nabla u_0\cdot \nabla \psi
\, (u_0)_{x_1}\, dx=
 \int_{\Omega} \chi\psi_{x_1}\, dx.
\end{equation}

Now let $\overline{B_s}(0)\subset\Omega$. Using that, by Lemma \ref{noweiss},
$\chi\equiv M$ in ${B_s}(0)\cap\{x_1>0\}$ and  $\chi\equiv \overline{M}$ in ${B_s}(0)\cap\{x_1<0\}$, for
a constant $\overline{M}$, with $\overline{M}=0$ or $\overline{M}=M$,  and the fact that
$\nabla u_0=\alpha\chi_{\{x_1>0\}}e_1$, we obtain for
$\psi\in C_0^{\infty}({B_s}(0))$
$$
-\int_{\{x_1>0\}} \frac{\alpha^{p_0}}{{p_0}} \psi_{x_1}\, dx + \int_{\{x_1>0\}}
{\alpha^{p_0}} \psi_{x_1} \, dx=M\int_{\{x_1>0\}}
 \psi_{x_1}+\overline{M}\int_{\{x_1<0\}}
 \psi_{x_1}.
$$
Then, integrating by parts, we get
$$
 \Big(-\frac{\alpha^{p_0}}{{p_0}} + {\alpha^{p_0}}\Big)\int_{\{x_1=0\}}
 \psi \, dx'=M\int_{\{x_1=0\}}
 \psi \, dx'-\overline{M}\int_{\{x_1=0\}}
 \psi\, dx'.
$$

Thus, $(-\frac{\alpha^{p_0}}{{p_0}} + {\alpha^{p_0}})=M-\overline{M}$. Since we have assumed that $\alpha>0$, it follows that
$\overline{M}=0$ and therefore, $\alpha=\Big(\frac{p_0}{p_0-1}\,M\Big)^{1/p_0}$.

\end{proof}
\end{section}
\begin{section}{Asymptotic behavior of limit functions}

In this section we analyze the behavior of limit functions near the free boundary.

\medskip

For the next result we will need the following definition

\begin{defi} Let $u$ be a continuous nonnegative function in a domain $\Omega\subset
\Bbb R^{N}$. Let $x_0\in\Omega\cap\partial\{u>0\}$. We say that
$x_0$ is a regular point from the positive side if there is a ball
$B\subset\{u>0\}$ with $x_0\in\partial B$.
\end{defi}

\begin{theo}
\label{clw1-Theorem 6.1}
Let $\uepj$, $\fepj$, $p_{\epj}$, $\ep_j$, $u$, $f$ and $p$ be as in Lemma \ref{clw1-Lemma 3.2}.

Let $x_0\in\Omega\cap\partial\{u>0\}$. Assume one of the following
conditions holds:

\begin{itemize}
\item[(D)] There exist $\gamma>0$ and $0<c<1$ such that,
for every $x\in B_\gamma(x_0)\cap\partial\{u>0\}$ which is regular from the positive side and
$r\le\gamma$, there holds that $|\{u=0\}\cap B_r(x)|\ge c|B_r(x)|$.
\item[(L)] There exist $\gamma>0$, $\theta>0$ and $s_0>0$ such that for every point
$y\in B_\gamma(x_0)\cap\partial\{u>0\}$ which is regular from the positive side,
and for every ball $B_r(z)\subset\{u>0\}$ with $y\in\partial B_r(z)$ and $r\le\gamma$,
there exists a  unit vector ${\tilde e}_y$,  with
$\langle{\tilde e}_y,z-y\rangle>\theta||z-y||$, such that
$u(y-s{\tilde e}_y)=0$ for $0< s< s_0.$
\end{itemize}

Then,
$$
\limsup_{\stackrel{x\to x_0}{u(x)>0}}\,|\nabla u(x)|=0
\qquad\mbox{or} \qquad \limsup_{\stackrel{x\to x_0}{u(x)>0}}\,|\nabla u(x)|=
\lambda^*(x_0),
$$
where $\lambda^*(x)=\Big(\frac{p(x)}{p(x)-1}\,M\Big)^{1/p(x)}$ and
$\int \beta(s)\, ds=M$.
\end{theo}

\begin{rema}
In \cite{LW5} we prove that if $\uepj$, $\fepj$, $p_{\epj}$, $\ep_j$, $u$, $f$ and $p$ are as  in Theorem \ref{clw1-Theorem 6.1},
with $\uepj$ local minimizers of an energy functional then, $u$ satisfies condition (D) in Theorem \ref{clw1-Theorem 6.1} at every point in
$\Omega\cap\partial\{u>0\}$.
\end{rema}

\begin{proof}[Proof of Theorem \ref{clw1-Theorem 6.1}]
Let
$$
\alpha:=\limsup_{\stackrel{x\to x_0}{u(x)>0}} |\nabla u(x)|.
$$
Since $u\in Lip_{\rm{loc}}(\Omega)$, $\alpha<\infty$. If,
$\alpha=0$ there is nothing to prove. So, suppose that $\alpha>0$.
By the definition of $\alpha$ there exists a sequence
$z_k\rightarrow x_0$ such that
$$
u(z_k)>0,\quad \quad |\nabla u(z_k)|\rightarrow \alpha.
$$
Let $y_k$ be the nearest point from $z_k$ to $\Omega \cap
\partial\{u>0\}$ and let $d_k = |z_k-y_k|$.

Consider the blow up sequence $u_{d_k}$ with respect to
$B_{d_k}(y_k)$.  This is, $u_{d_k}(x)=\frac 1{d_k}u(y_k+d_k x)$.
Since $u$ is locally Lipschitz, and $u_{d_k}(0)=0$ for every $k$, there
exists $u_0\in Lip(\R^N)$, such  that (for a subsequence)
$u_{d_k}\to u_0$ uniformly on compact sets of $\R^N$.

Since $\Delta_{p(x)} u = f$ in  $\{u>0\}$, by interior H\"older gradient estimates (see, for instance,
\cite{Fan}), we
have that  $\Delta_{p_0} u_0=0$ in $\{u_0>0\}$ with $p_0=p(x_0)$.

Now, set $\bar{z}_k=(z_k-y_k)/d_k\in \partial B_1$. We may assume
that  $\bar{z}_k\to \bar{z}\in \partial B_1$. Take
$$
\nu_k:=\frac{\nabla u_{d_k}(\bar{z}_k)}{|\nabla
u_{d_k}(\bar{z}_k)|}= \frac{\nabla u({z}_k)}{|\nabla u({z}_k)|}.
$$
For a subsequence, and after a rotation, we can assume that
$\nu_k\to e_1$. Observe that $B_{2/3}(\bar{z})\subset
B_1(\bar{z}_k)$ for $k$ large, and therefore $\Delta_{p_0}u_0=0$ there. By interior H\"older gradient estimates, we have $\nabla u_{d_k}\to \nabla u_0$ uniformly in
$B_{1/3}(\bar{z})$, and therefore $\nabla u(z_k)\to \nabla
u_0(\bar z)$. Thus, $\nabla u_0(\bar z)=\alpha \,e_1$ and, in
particular, $\partial_{x_1} u_0(\bar{z})=\alpha$.

Next, we claim that $|\nabla u_0|\leq \alpha$ in $\R^N$. In fact,
let $R>1$ and $\delta>0$. Then, there exists $\tau_0>0$ such that
$|\nabla u(x)|\leq \alpha+\delta$ for any $x\in B_{\tau_0
R}(x_0)$. For $|z_k-x_0|<\tau_0 R/2$ and $d_k<\tau_0/2$ we have
$B_{d_k R}(z_k)\subset B_{\tau_0 R}(x_0)$ and therefore, $|\nabla
u_{d_k}(x)|\leq \alpha +\delta$ in $B_R$ for $k$ large. Passing to
the limit, we obtain $|\nabla u_0|\leq \alpha+\delta$ in $B_R$,
and since $\delta$ and $R$ were arbitrary, the claim holds.

Since $\nabla u_0$ is H\"older continuous in $B_{1/3}(\bar{z})$,
there holds that $\nabla u_0\neq0$ in a neighborhood of $\bar z$.
Thus,  $u_0\in W^{2,2}$ in a ball $B_r(\bar z)$ for some $r>0$
(see, for instance, \cite{Tolk} or \cite{CL}) and, since
$$
\int |\nabla u_0|^{p_0-2}\nabla u_0\cdot\nabla \varphi\,dx=0 \quad\mbox{for every
}\varphi\in C_0^\infty(B_r(\bar z)),
$$
taking $\psi\in C_0^\infty(B_r(\bar z))$ and $\varphi=\psi_{x_1}$,
and integrating by parts we see that, for $w=\frac{\partial
u_0}{\partial x_1}$ and suitable coefficients $a_{ij}(\nabla u_0)$,
$$
\sum_{i,j=1}^N\int_{B_r(\bar z)} a_{ij}\big(\nabla
u_0(x)\big)w_{x_j}\psi_{x_i}\,dx=0.
$$
This is, $w$ is a solution to a uniformly elliptic equation
$$
\mathcal T w:= \sum_{i,j=1}^N\frac{\partial}{\partial
x_i}\Big(a_{ij}\big(\nabla u_0(x)\big)w_{x_j}\Big)=0.
$$

Now, since $w\le \alpha$ in $B_r(\bar z)$,
$w(\bar z)=\alpha$ and $\mathcal T w=0$ in $B_r(\bar z)$, by the strong maximum principle
we conclude that $w\equiv \alpha$ in $B_r(\bar z)$.

 Now, since we can repeat this argument around any point where $w=\alpha$,  by a
continuation argument, we have that $w=\alpha$ in $B_1(\bar{z})$.

Therefore, $\nabla u_0=\alpha\, e_1$ in $B_1(\bar{z})$ and we
have, for some $y\in \R^N$, $u_0(x)=\alpha (x_1-y_1)$ in
$B_1(\bar{z})$. Since $u_0(0)=0$, there holds that $y_1=0$ and
$u_0(x)=\alpha x_1$ in $B_1(\bar{z})$. Finally, since
$\Delta_{p_0} u_0=0$ in $\{u_0>0\}$ by a continuation argument we
have that $u_0(x)=\alpha x_1$ in $\{x_1\geq 0\}$.

On the other hand, as $u_0\geq 0$, $\Delta_{p_0}u_0=0$ in $\{u_0>0\}$ and
$u_0=0$ in $\{x_1=0\}$ we have, by Lemma \ref{development11}, that
$$u_0(x)=-\bar\alpha x_1+o(|x|) \quad \mbox{ in } \{x_1<0\}$$
for some $\bar\alpha \geq 0$.

Now, define for $\lambda>0$, $(u_0)_{\lambda}(x)=\frac{1}{\lambda}
u_0(\lambda x)$. There exist a sequence $\lambda_n\to 0$ and
$u_{00}\in Lip(\R^N)$ such that $(u_0)_{\lambda_n}\to u_{00}$
uniformly on compact sets of $\R^N$. We have $u_{00}(x)=\alpha
x_1^++\bar\alpha x_1^-$.

We will show that $\bar\alpha=0$.

In fact, first assume condition (D) holds. We observe that, for any $R$, there holds for large $k$, that

 $$|\{u=0\}\cap B_{d_kR}(y_k)|\ge c|B_{d_kR}(y_k)|,$$
 implying that
$$|\{u_{d_k}=0\}\cap B_{R}(0)|\ge c|B_{R}(0)|,$$
and therefore
$$|\{u_0=0\}\cap B_{R}(0)|\ge c|B_{R}(0)|, \qquad {\rm and} \qquad |\{u_{00}=0\}\cap B_{1}(0)|\ge c|B_{1}(0)|.$$

This shows that $\bar\alpha=0$.

Now assume condition (L) holds. Then, for every $k$ there exists a unit vector $\tilde e_k$ such that
$$\langle{\tilde e_k},\frac{z_k-y_k}{d_k}\rangle>\theta \quad\text{and}\quad u(y_k-sd_k{\tilde e}_k)=0 \quad\text{ for }\quad 0< s< s_0.$$
So that
$$u_{d_k}(-s\tilde e_k)=0\quad\text{ for }\quad 0< s< s_0.$$

For a subsequence we have $\tilde e_k\to \tilde e$, and $\frac{z_k-y_k}{d_k}\to\bar z$,  with $\langle{\tilde e},\bar z\rangle\geq\theta$, implying
that $u_0(-s\tilde e)=0$ for $0< s< s_0$ and thus, $u_{00}(-\tilde e)=0$.

We now observe that, since we have seen that $B_1(\bar z)\subset\{u_0(x)=\alpha x_1\}=\{x_1>0\}$ and $0\in\partial B_1(\bar z)$, it follows that
$\bar z=e_1$.
Therefore $0=u_{00}(-\tilde e)=\bar\alpha \langle{\tilde e},e_1\rangle\ge\bar\alpha\theta$.

So that $\bar\alpha=0$ under condition (L) as well.

Now, by Lemma \ref{clw1-Lemma 3.2} we see that there exists a sequence $\delta_n\to
0$ and solutions  $u^{\delta_n}$  to
$P_{\delta_n}(f^{\delta_n}, p_{\delta_n})$ such that  $u^{\delta_n}\to
u_0$ uniformly on compact sets of $\R^N$, with
$f^{\delta_n}\rightharpoonup 0$ $*-$weakly in $L^\infty$ on
compact sets of $\R^N$, $p_{\delta_n}\to p(x_0)$ uniformly on
compact sets of $\R^N$ and
$\|\nabla p_{\delta_n}\|_{L^{\infty}}\to 0$ on
compact sets of $\R^N$.

 Applying a second time Lemma \ref{clw1-Lemma 3.2} we find a
sequence $\tilde\delta_n\to 0$ and solutions  $u^{\tilde\delta_n}$ to $P_{\tilde\delta_n}(f^{\tilde\delta_n},
p_{\tilde\delta_n})$ such that  $u^{\tilde\delta_n}\to u_{00}$
uniformly on compact sets of $\R^N$, with
$f^{\tilde\delta_n}\rightharpoonup 0$ $*-$weakly in $L^\infty$ on
compact sets of $\R^N$, $p_{\tilde\delta_n}\to p(x_0)$ uniformly on
compact sets of $\R^N$ and
$\|\nabla p_{\tilde\delta_n}\|_{L^{\infty}}\to 0$ on
compact sets of $\R^N$. Now we
can apply Proposition \ref{clw1-Proposition 5.2} and we conclude
that $\alpha=\lambda^*(x_0)$.
\end{proof}

\begin{defi}\label{nondegener} Let $v$ be a continuous nonnegative function
in a domain $\Omega\subset\mathbb{R}^N$. We say that $v$ is
nondegenerate at a point  $x_0\in \Omega\cap\{v=0\}$ if there
exist $c>0$, $r_0>0$ such that one of the following conditions holds:
\begin{equation}\label{nond-prom-bol}
\fint_{B_r(x_0)} v\, dx\geq c r\quad \mbox{ for } 0<r\leq r_0,
\end{equation}
\begin{equation}\label{nond-prom-casc}
\fint_{\partial B_r(x_0)} v\, dx\geq c r\quad \mbox{ for } 0<r\leq r_0,
\end{equation}
\begin{equation}\label{nond-sup}
\sup_{B_r(x_0)} v\geq c r\quad \mbox{ for } 0<r\leq r_0.
\end{equation}

\medskip

We say that $v$ is uniformly nondegenerate on a set
$\Gamma\subset\Omega\cap\{v=0\}$ in the sense of \eqref{nond-prom-bol} (resp. \eqref{nond-prom-casc}, \eqref{nond-sup}) if the constants $c$ and $r_0$
in \eqref{nond-prom-bol} (resp.  \eqref{nond-prom-casc}, \eqref{nond-sup}) can be taken independent of the point $x_0\in\Gamma$.
\end{defi}

\begin{rema}\label{equiv-nondeg}
Assume $v\ge 0$ is locally Lipschitz continuous in a domain $\Omega\subset\mathbb{R}^N$, $v\in W^{1,p(\cdot)}(\Omega)$ with $\Delta_{p(x)} v \ge  f \chi_{\{v >0\}}$,
where $f\in L^{\infty}(\Omega)$, $1<p_{\min}\le p(x)\le p_{\max}<\infty$ and $p(x)$ is
Lipschitz continuous. Then the three concepts of nondegeneracy in Definition \ref{nondegener} are equivalent
(for the idea of the proof, see Remark 3.1 in \cite{LW1}, where  the case $p(x)\equiv 2$ and $f\equiv 0$ is treated).
\end{rema}

\begin{rema}
In \cite{LW5} we prove that if $\uepj$, $\fepj$, $p_{\epj}$, $\ep_j$, $u$, $f$ and $p$ are as  in Lemma \ref{clw1-Lemma 3.2},
with $\uepj$ local minimizers of an energy functional then, $u$ is locally uniformly nondegenerate on
$\Omega\cap\partial\{u>0\}$.
\end{rema}

\begin{theo}\label{limsup4}
Let $\uepj$, $\fepj$, $p_{\epj}$, $\ep_j$, $u$, $f$ and $p$ be as in Lemma \ref{clw1-Lemma 3.2}.

Let
$x_0\in  \Omega\cap\partial\{u>0\}$ and suppose that $u$ is uniformly nondegenerate on $\Omega\cap\partial\{u>0\}$ in a neighborhood of $x_0$. Assume there is  a ball $B$ contained in
$ \{u=0\} $ touching  $ x_0$, then
\begin{equation}\label{limsup224}
\limsup_{\stackrel{x\to x_0}{u(x)>0}} \frac{u(x)}{\mbox{dist}(x,B)}=  \lambda^*(x_0),\end{equation} where
$\lambda^*(x)=\Big(\frac{p(x)}{p(x)-1}\,M\Big)^{1/p(x)}$ and $\int
\beta(s)\, ds=M$.
\end{theo}

\begin{proof}
 Let $\ell$ be the finite limit on the left hand side of \eqref{limsup224} and let $y_k \to x_0$
with $u(y_k)>0$ be such that
$$\frac{u(y_k)}{d_k}\to \ell, \quad d_k=\mbox{dist}(y_k,B).$$
Consider the blow up sequence $u_k$ with respect to
$B_{d_k}(x_k)$, where $x_k\in\partial B$ are points with
$|x_k-y_k|=d_k$, this is, $u_k(x)=\frac{u(x_k+d_k x)}{d_k}$.
Choose a subsequence with blow up limit
$u_0$, such that there exists
$$e:=\lim_{k\to\infty} \frac{y_k-x_k}{d_k}.$$

As in Theorem \ref{clw1-Theorem 6.1}, we see that $\Delta_{p_0} u_0=0$ in $\{u_0>0\}$ with $p_0=p(x_0)$.

By construction, $u_0(e)=\ell=\ell\langle e,e\rangle$,
$u_0(x)\leq \ell\langle x, e\rangle $ for $\langle x ,
e\rangle\geq 0$, $u_0(x)=0$ for $\langle x, e\rangle\leq 0.$

Let us see that $\ell>0$. In fact, if $\ell=0$, then $u_0\equiv 0$. Since $u(y_k)>0$ and $u(x_k)=0$, there exists
$z_k\in\partial\{ u>0\}$ in the segment between $y_k$ and $x_k$. By the nondegeneracy assumption,
$$\sup_{B_r(z_k)} u\ge cr\quad \text{for } 0< r\le r_0, \ c>0$$
and, in particular,
$$\sup_{B_{d_k}(z_k)}u \ge cd_k\quad \text{for } k\ge k_0.$$
Then, there exists $a_k$ such that $|a_k-z_k|\le d_k$ and $u(a_k)\ge cd_k$. Then, letting $\bar x_k=\frac{a_k-x_k}{d_k}$, we get that
$u_k(\bar x_k)\ge c$, with $|\bar x_k|\le 2$.  It follows that there exists $\bar x$ with $|\bar x|\le 2$ such that $u_0(\bar x)\ge c>0$, which is
a contradiction.

We now observe that $\nabla u_0(e)=\ell\,e$, and thus, $|\nabla
u_0(e)|=\ell>0$. Using that $\nabla u_0$ is continuous in
$\{u_0>0\}$ we deduce, from the fact that $\Delta_{p_0} u_0=0$ in
$\{u_0>0\}$, that $u_0\in W_{\rm{loc}}^{2,2}$ in
$\{u_0>0\}\cap\{|\nabla u_0|>0\}$. Then, $u_0$ is a  solution of
$Lv=0$ in $\{u_0>0\}\cap \{|\nabla u_0|>0\}$ where
$$
Lv:=\sum_{i,j=1}^N b_{ij}(\nabla u_0)v_{x_ix_j}
$$
is the uniformly elliptic operator given by
$$b_{ij}(z)=\delta_{ij}+ \frac{(p_0-2)}{|z|^2}{z_i z_j}.$$

Since $w(x)=\ell\langle x, e\rangle$ also satisfies $Lw=0$ we have,  from the strong maximum principle,
that $u_0$ and $w$ must coincide in a neighborhood of the point $e$.

By continuation we have that $u_0(x)=\ell\langle x, e\rangle^+$.
Thus, applying Lemma \ref{clw1-Lemma 3.2} as we did in Theorem \ref{clw1-Theorem 6.1} and using  Proposition \ref{clw1-Proposition 5.2}, we get that
$\ell=\lambda^*(x_0)$.
\end{proof}

\medskip

\begin{defi}
\label{clw2-Definition 3.1} We say that $\nu$ is the inward unit
normal to the free boundary $\fb$ at a point $x_0\in\fb$ in the
 measure theoretic sense, if $\nu
\in\Bbb R^N$, $|\nu|=1$
and
\begin{equation}
\label{(3.3)} \lim_{r\to 0} \frac1{r^{N}} \int_{B_r(x_0)}
|\chi_{\{u>0\}}- \chi_{\{x\,/\, \langle x-x_0,\nu\rangle
>0\}}|\,dx = 0.
\end{equation}

\end{defi}

\medskip

\begin{theo}
\label{clw2-Theorem 3.1}
Let $\uepj$, $\fepj$, $p_{\epj}$, $\ep_j$, $u$, $f$ and $p$ be as in Lemma \ref{clw1-Lemma 3.2}.

Let $x_0\in \Omega\cap\partial\{u>0\}$ be such that $\fb$ has at
$x_0$ an  inward unit normal $\nu$ in the measure theoretic sense
and suppose that $u$ is nondegenerate at $x_0$. Assume, in addition, that either condition (D) or condition (L) in Theorem \ref{clw1-Theorem 6.1} holds at $x_0$. Then,
$$
u(x)=\lambda^*(x_0)\langle x-x_0, \nu\rangle^+  + o(|x-x_0|),
$$
where $\lambda^*(x)=\Big(\frac{p(x)}{p(x)-1}\,M\Big)^{1/p(x)}$ and
$\int \beta(s)\, ds=M$.
\end{theo}

\begin{proof}
Assume that $x_0=0$, and $\nu=e_1$. Take
$u_{\lambda}(x)=\frac{1}{\lambda} u(\lambda x).$ Let $\rho>0$ such
that $B_{\rho}\subset\subset\Omega$. Since $u_{\lambda}\in
Lip(B_{\rho/\lambda})$ uniformly in $\lambda$, $u_{\lambda}(0)=0$,
there exist $\lambda_j\to 0$ and $U\in Lip(\R^ N)$ such that
$u_{\lambda_j}\to U$ uniformly on compact sets of $\R^N$. {}From
Lemma \ref{clw1-Lemma 3.1}, $\Delta_{p(\lambda x)} u_{\lambda}=\lambda f(\lambda x)$
in $\{u_{\lambda}>0\}$. Using the fact that $e_1$ is the
 inward normal in the measure theoretic sense, we have, for
fixed $k$,
$$|\{u_{\lambda}>0\}\cap\{x_1<0\}\cap B_k|\to 0 \quad \mbox{ as } \lambda \to 0.$$
Hence, $U=0$ in $\{x_1<0\}$. Moreover, $U$ is nonnegative in
$\{x_1>0\}$, $\Delta_{p_0} U=0$ in $\{U>0\}$ with $p_0=p(x_0)$ and $U$ vanishes in $\{x_1\leq
0\}$. Then, by Lemma \ref{development11} we have that there
exists $\alpha\geq 0$ such that
$$U(x)=\alpha x_1^++o(|x|).$$

By Lemma \ref{clw1-Lemma 3.2} we see that there exist a sequence $\delta_n\to
0$ and solutions $u^{\delta_n}$ to
$P_{\delta_n}(f^{\delta_n}, p_{\delta_n})$ such that $u^{\delta_n}\to
U$ uniformly on compact sets of $\R^N$, with
$f^{\delta_n}\rightharpoonup 0$ $*-$weakly in $L^\infty$ on
compact sets of $\R^N$, $p_{\delta_n}\to p(x_0)$ uniformly on
compact sets of $\R^N$ and
$\|\nabla p_{\delta_n}\|_{L^{\infty}}\to 0$ on
compact sets of $\R^N$.

Define
$U_{\lambda}(x)=\frac{1}{\lambda} U(\lambda x)$, then
$U_{\lambda}\to \alpha x_1^+$ uniformly on compact sets of
$\R^N$.  Applying a second time Lemma \ref{clw1-Lemma 3.2} we find a
sequence $\tilde\delta_n\to 0$ and solutions  $u^{\tilde\delta_n}$ to
$P_{\tilde\delta_n}(f^{\tilde\delta_n}, p_{\tilde\delta_n})$ such that  $u^{\tilde\delta_n}\to \alpha x_1^+$
uniformly on compact sets of $\R^N$, with
$f^{\tilde\delta_n}\rightharpoonup 0$ $*-$weakly in $L^\infty$ on
compact sets of $\R^N$, $p_{\tilde\delta_n}\to p(x_0)$ uniformly on
compact sets of $\R^N$ and
$\|\nabla p_{\tilde\delta_n}\|_{L^{\infty}}\to 0$ on
compact sets of $\R^N$.

By the nondegeneracy assumption on $u$, we have
$$\frac{1}{r^N} \int_{B_r} u_{\lambda_j} \, dx \geq cr$$ and then
$$\frac{1}{r^N} \int_{B_r} U_{\lambda_j} \, dx \geq cr.$$

Therefore $\alpha>0$. Now, by Proposition \ref{clw1-Proposition 5.2}, $\alpha=\lambda^*(x_0).$

We have shown that
$$U(x)=\begin{cases} \lambda^*(x_0) x_1+o(|x|) &\quad x_1>0\\
0 &\quad x_1\leq 0.
       \end{cases}
$$

Then, using that $\Delta_{p(\lambda x)} u_{\lambda}=\lambda
f(\lambda x)$ in $\{u_{\lambda}>0\}$, by interior H\"older
gradient estimates we have $\nabla u_{\lambda_j}\rightarrow \nabla
U$ uniformly on compact subsets of $\{U>0\}$. Then, by Theorem
\ref{clw1-Theorem 6.1}, $|\nabla U|\leq \lambda^*(x_0)$ in $\R^N$.
As $U=0$ on $\{x_1=0\}$ we have, $U\leq \lambda^*(x_0)x_1$ in
$\{x_1>0\}$.

We claim that either $U\equiv \lambda^*(x_0)x_1$ in $\{x_1>0\}$ or else $U < \lambda^*(x_0)x_1$  in $\{x_1>0\}$.

In fact, if there exists $\bar x$ with $\bar x_1>0$ such that the equality holds at $\bar x$, then we proceed exactly as we did in the proof of Theorem \ref{limsup4} and deduce, from the strong maximum principle, that equality holds in a neighborhood of $\bar x$. Then, by continuation, we get $U\equiv \lambda^*(x_0)x_1$ in $\{x_1>0\}$.

So let us now assume  that $U < \lambda^*(x_0)x_1$  in $\{x_1>0\}$. Let $\delta >0$ be such that $U(\delta e_1)>0$.
Let $w$ be such that
$$\begin{cases} \Delta_{p_0} w =0 &\quad \mbox{ in } B_{\delta}^+\\
w =0 & \quad \mbox{ on }\{x_1=0\}\\
w=U &\quad \mbox{ on }\partial B_{\delta}\cap\{x_1>0\}.
\end{cases}
$$
Since $\Delta_{p_0} U\ge 0$ (this follows, for instance, from the application of Lemma \ref{radpos} with $g=0$ and $p(x)=p_0$), we have that $w\ge U$ in $B_{\delta}^+$. Therefore $w\ge \lambda^*(x_0) x_1+o(|x|)$ in  $B_{\delta}^+$.

We also have $w\le \lambda^*(x_0) x_1$ in  $B_{\delta}^+$. Moreover, $w< \lambda^*(x_0) x_1$ in  $B_{\delta}^+$, because this holds on $\partial B_{\delta}\cap\{x_1>0\}$, and with the same argument employed above we can see  that, if equality holds at a point in $B_{\delta}^+$, then it must hold everywhere on $B_{\delta}^+$.

On the other hand, we know that $w\in C^{1, \alpha}(\overline {B_\sigma^+})$ for any $\sigma <\delta$, and since
$w\ge \lambda^*(x_0) x_1+o(|x|)$ in  $B_{\delta}^+$, then $|\nabla w(0)|>0$, implying that $|\nabla w|>0$ in
$\overline{ B_{\gamma}^+}$ for some  $\gamma>0$.

Since, in $B_{\gamma}^+$, both $w$ and $\lambda^*(x_0) x_1$ are solutions to $Lv=0$, with $L$ a linear uniformly elliptic operator in nondivergence form, with $w< \lambda^*(x_0) x_1$ in $B_{\gamma}^+$, from the Hopf's boundary principle we get that $w\le (\lambda^*(x_0)-\rho) x_1+o(|x|)$ for some $\rho>0$ in $B_{\gamma}^+$. This is in contradiction with the fact that $w\ge \lambda^*(x_0) x_1+o(|x|)$ in  $B_{\delta}^+$.

This shows that $U\equiv \lambda^*(x_0)x_1$ in $\{x_1>0\}$. The proof is complete.
\end{proof}

\end{section}



\begin{section}{Weak solutions to the free boundary problem $P(f,p,{\lambda}^*)$}
\label{sect-weak-solut}

In this section we give a notion of
weak solution to the free boundary problem \ref{bernoulli-px} and
we show that, under suitable assumptions, limit functions to
problems $\pep(\fep, p_{\ep})$  are weak  solutions, in this sense, to the free
boundary problem with
$\lambda^*(x)=\Big(\frac{p(x)}{p(x)-1}\,M\Big)^{1/p(x)}$, $p=\lim
p_\ep$ and $f=\lim \fep$.

As a consequence, we are able to apply to limit functions the
result on the  regularity of the free boundary we prove in
\cite{LW4} (see Theorem \ref{reg-weak} below).

\begin{defi}\label{weak2} Let $\Omega\subset \Bbb R^{N}$ be a domain. Let $p$ be a measurable function in $\Omega$ with
$1<p_{\min}\le p(x)\le p_{\max}<\infty$, $\lstar$ continuous in
$\Omega$ with $0<\lone\le\lstar(x)\le\ltwo<\infty$ and $f\in
L^\infty(\Omega)$. We call $u$ a weak solution of
\ref{bernoulli-px} in $\Omega$ if
\begin{enumerate}
\item $u$ is continuous and nonnegative in $\Omega$, $u\in W^{1,p(\cdot)}(\Omega)$ and
$\Delta_{p(x)}u=f$ in $\Omega\cap\{u>0\}$.
\item For
$D\subset\subset \Omega$ there are constants $0< c_{\min}\leq
C_{\max}$ and $r_0>0$ such that for balls $B_r(x)\subset D$ with $x\in
\partial \{u>0\}$ and $0<r\le r_0$
$$
c_{\min}\leq \frac{1}{r}\sup_{B_r(x)} u \leq C_{\max}.
$$
\item For $\mathcal{H}^{N-1}$ a.e.
$x_0\in\partial_{\rm{red}}\{u>0\}$ (this is, for ${\mathcal
H}^{N-1}$-almost every point $x_0$ such that $\partial\{u> 0\}$
has an exterior  unit normal
 $\nu(x_0)$ in the measure theoretic sense)
$u$ has the asymptotic development
\begin{equation}\label{asymp-w}
u(x)=\lambda^*(x_0)\langle x-x_0,\nu(x_0)\rangle^-+o(|x-x_0|).
\end{equation}

\item For every $ x_0\in \Omega\cap\partial\{u>0\}$,
\begin{align*}
& \limsup_{\stackrel{x\to x_0}{u(x)>0}} |\nabla u(x)| \leq
\lambda^*(x_0).
\end{align*}

If there is a ball $B\subset\{u=0\}$ touching
$\Omega\cap\partial\{u>0\}$ at $x_0$, then
$$\limsup_{\stackrel{x\to x_0}{u(x)>0}} \frac{u(x)}{\mbox{dist}(x,B)}\geq  \lambda^*(x_0). $$
\end{enumerate}
\end{defi}

{}From the definition of weak solution above, and the results in
the previous sections we obtain:

\begin{theo}\label{lim=weak}
Let $\uepj$, $\fepj$, $p_{\epj}$, $\ep_j$, $u$, $f$ and $p$ be as in Lemma \ref{clw1-Lemma 3.2}.

Assume that $u$ is locally uniformly nondegenerate on  $\Omega\cap\partial\{u>0\}$ and that at every point  $x_0\in\Omega\cap\partial\{u>0\}$
either condition (D) or condition (L) in Theorem \ref{clw1-Theorem 6.1} holds. Then,  $u$ is a weak solution to the free boundary problem: $u\ge0$ and
\begin{equation}
\tag{$P(f,p,{\lambda}^*)$}
\begin{cases}
\Delta_{p(x)}u = f & \mbox{in }\{u>0\}\\
u=0,\ |\nabla u| = \lambda^*(x) & \mbox{on }\partial\{u>0\}
\end{cases}
\end{equation}
with $\lambda^*(x)=\Big(\frac{p(x)}{p(x)-1}\,M\Big)^{1/p(x)}$ and
$M=\int \beta(s)\, ds$.
\end{theo}

\begin{proof} The result follows from Theorem \ref{estim-lip}, Lemma \ref{clw1-Lemma 3.1}, Remark \ref{equiv-nondeg} and Theorems \ref{clw1-Theorem 6.1}, \ref{limsup4} and \ref{clw2-Theorem 3.1}.
\end{proof}

\medskip

\begin{rema}
In \cite{LW5} we prove that if $\uepj$, $\fepj$, $p_{\epj}$, $\ep_j$, $u$, $f$ and $p$ are as  in Lemma \ref{clw1-Lemma 3.2},
with $\uepj$ local minimizers of an energy functional,  $u$ is under the assumptions of Theorem~\ref{lim=weak}.
\end{rema}

\medskip

In \cite{LW4} we prove the following result for weak solutions
that applies, in particular, to limit functions $u$ as those in
Theorem \ref{lim=weak}, at every point in
$\Omega\cap\partial_{\rm{red}}\{u>0\}$.
\begin{theo} \label{reg-weak} Let $p\in Lip(\Omega)$ and $\lambda^*$ H\"older continuous in $\Omega$.
Let $u$ be a weak solution of \ref{bernoulli-px} in $\Omega$. Let
$x_0\in\Omega\cap\partial_{\rm{red}}\{u>0\}$ be such that $u$ has
the asymptotic development \eqref{asymp-w}. There exists $r_0>0$
such that $B_{r_0}(x_0)\cap\partial\{u>0\}$ is a $C^{1,\alpha}$
surface for some $0<\alpha<1$. It follows that, in $B_{r_0}(x_0)$,
$u$ is $C^1$ up to $\partial\{u>0\}$ and the free boundary
condition is satisfied in the classical sense. In addition, there
is a neighborhood $\mathcal U$ of $x_0$ such that $\nabla u\neq0$
in $\mathcal U\cap\{u>0\}$, $u\in W_{\rm{loc}}^{2,2}(\mathcal
U\cap\{u>0\})$ and the equation is satisfied in a pointwise sense
in $\mathcal U\cap\{u>0\}$. If moreover $\nabla p$ and $f$ are
H\"older continuous in $\Omega$, then $u\in C^2(\mathcal
U\cap\{u>0\})$ and the equation is satisfied in the classical
sense in $\mathcal U\cap\{u>0\}$.
\end{theo}

\bigskip

\end{section}


\appendix

\section{} \label{appA1}

\setcounter{equation}{0}

In this appendix we collect some result on Lebesgue and Sobolev spaces with variable exponent as well as some other results that are used in the paper.

\medskip

Let $p :\Omega \to  [1,\infty)$ be a measurable bounded function,
called a variable exponent on $\Omega$ and denote $p_{\max} =
{\rm ess sup} \,p(x)$ and $p_{\min} = {\rm ess inf} \,p(x)$. We define the
variable exponent Lebesgue space $L^{p(\cdot)}(\Omega)$ to consist
of all measurable functions $u :\Omega \to \R$ for which the
modular $\varrho_{p(\cdot)}(u) = \int_{\Omega} |u(x)|^{p(x)}\, dx$
is finite. We define the Luxemburg norm on this space by
$$
\|u\|_{L^{p(\cdot)}(\Omega)} = \|u\|_{p(\cdot)}  = \inf\{\lambda > 0: \varrho_{p(\cdot)}(u/\lambda)\leq 1 \}.
$$

This norm makes $L^{p(\cdot)}(\Omega)$ a Banach space.

One central property of these spaces (since $p$ is bounded) is that $\varrho_{p(\cdot)}(u_i)\to 0$
 if and only if $\|u_i\|_{p(\cdot)}\to 0$, so that the norm and modular topologies coincide. In fact, we have

\bigskip
\begin{prop}\label{equi}
There holds
\begin{align*}
\min\Big\{\Big(\int_{\Omega} |u|^{p(x)}\, dx\Big)
^{1/{p_{\min}}},& \Big(\int_{\Omega} |u|^{p(x)}\, dx\Big)
^{1/{p_{\max}}}\Big\}\le\|u\|_{L^{p(\cdot)}(\Omega)}\\
 &\leq  \max\Big\{\Big(\int_{\Omega} |u|^{p(x)}\, dx\Big)
^{1/{p_{\min}}}, \Big(\int_{\Omega} |u|^{p(x)}\, dx\Big)
^{1/{p_{\max}}}\Big\}.
\end{align*}
\end{prop}

\bigskip

Let $W^{1,p(\cdot)}(\Omega)$ denote the space of measurable functions $u$ such that $u$ and the distributional derivative $\nabla u$ are in $L^{p(\cdot)}(\Omega)$. The norm

$$
\|u\|_{1,p(\cdot)}:= \|u\|_{p(\cdot)} + \| |\nabla u| \|_{p(\cdot)}
$$
makes $W^{1,p(\cdot)}$ a Banach space.

The space $W_0^{1,p(\cdot)}(\Omega)$ is defined as the closure of the $C_0^{\infty}(\Omega)$ in $W^{1,p(\cdot)}(\Omega)$.

In some occasions, it is necessary  to assume extra hypotheses on the regularity of $p(x)$.
We say that $p$ is log-H\"{o}lder continuous
if there exists a constant $C$ such that
$$|p(x) - p(y)| \leq \frac{C}{\big|\log\, |x - y|\big|}$$
if $|x - y| < 1/2 $.

If one assumes that $p$ is log-H\"{o}lder continuous then, there holds that $C^{\infty}(\overline{\Omega})$ is dense in $W^{1,p(\cdot)}(\Omega)$.

Some important results for these spaces are

\begin{theo}\label{ref}
Let $p'(x)$ such that $$\frac{1}{p(x)}+\frac{1}{p'(x)}=1.$$ Then
$L^{p'(\cdot)}(\Omega)$ is the dual of $L^{p(\cdot)}(\Omega)$.
Moreover, if $p_{\min}>1$, $L^{p(\cdot)}(\Omega)$ and
$W^{1,p(\cdot)}(\Omega)$ are reflexive.
\end{theo}

\begin{theo}\label{imb}
Let $q(x)\leq p(x)$, then
 $L^{p(\cdot)}(\Omega)\hookrightarrow L^{q(\cdot)}(\Omega)$
continuously.
\end{theo}

We also have the following H\"older's inequality

\begin{theo} \label{holder}
Let $p'(x)$ be as in Theorem \ref{ref}. Then there holds
$$
\int_{\Omega}|f||g|\,dx \le 2\|f\|_{p(\cdot)}\|g\|_{p'(\cdot)},
$$
for all $f\in L^{p(\cdot)}(\Omega)$ and $g\in L^{p'(\cdot)}(\Omega)$.
\end{theo}

The following version of Poincare's inequality holds

\begin{theo}\label{poinc} Let $\Omega$ be bounded. Assume that $p(x)$ is log-H\"older continuous  in $\Omega$. For
every $u\in W_0^{1,p(\cdot)}(\Omega)$, the inequality
$$
\|u\|_{L^{p(\cdot)}(\Omega)}\leq C\|\nabla u\|_{L^{p(\cdot)}(\Omega)},
$$
holds with a constant $C$ depending on N, $\rm{diam}(\Omega)$ and the log-H\"older modulus of continuity of $p(x)$.
\end{theo}

\smallskip

For the proof of these results and more about these spaces, see
\cite{DHHR,KR} and the references therein.

\bigskip

\begin{rema}\label{desip} For any $x\in\Omega$, $\xi,\eta\in\R^N$ fixed we have the following inequalities
\begin{align*}
&|\eta-\xi|^{p(x)}\leq C (|\eta|^{p(x)-2} \eta-|\xi|^{p(x)-2} \xi)
(\eta-\xi)&\quad  \mbox{ if } p(x)\geq 2,\\
&
 |\eta-\xi|^2\Big(|\eta|+|\xi|\Big)^{p(x)-2}
\leq C (|\eta|^{p(x)-2} \eta-|\xi|^{p(x)-2} \xi)
(\eta-\xi)&\quad  \mbox{ if } p(x)< 2.\\
\end{align*}
These inequalities imply that the function $A(x,\xi)=|\xi|^{p(x)-2}\xi$ is strictly monotone. Then, the comparison principle for the $p(x)$-Laplacian holds since it follows from the monotonicity of $A(x,\xi)$.

\end{rema}

We will also need

\begin{lemm}\label{development11} Let $1<p_0<+\infty$.
Let $u$ be  Lipschitz continuous in $\overline{B_1^+}$, $u\geq 0$
in $B_1^+$, $\Delta_{p_0} u=0$ in $\{u>0\}$ and $u=0$ on $\{x_N=0\}$. Then,
in $B_1^+$ $u$ has the asymptotic development
$$
u(x)=\alpha x_N+ o(|x|),
$$
with $\alpha\ge 0$.
\end{lemm}
\begin{proof}
See \cite{CLW1} for $p_0=2$, \cite{DPS} for $1<p_0<+\infty$ and \cite{MW2} for a more general operator.
\end{proof}




\end{document}